\documentclass{conm-p-l}
\usepackage{amssymb}
\usepackage{amsmath}
\usepackage{amsfonts}

\newtheorem{theorem}{Theorem}[section]
\newtheorem{lemma}[theorem]{Lemma}
\theoremstyle{definition}
\newtheorem{definition}[theorem]{Definition}

\theoremstyle{remark}
\newtheorem{remark}[theorem]{Remark}
\numberwithin{equation}{section}
\theoremstyle{plain}

\newtheorem{condition}{Condition}

\newtheorem{corollary}{Corollary}

\newtheorem{proposition}{Proposition}

\copyrightinfo{2013}{C. Durastanti, X. Lan}
\input{tcilatex}

\begin{document}
\title[High-Frequency Tail Index Estimation by Nearly Tight Frames]{%
High-Frequency Tail Index Estimation by Nearly Tight Frames}
\author{Claudio Durastanti}
\thanks{ Research partially supported by ERC Grant n. 277742 Pascal}
\address[C. Durastanti]{ University of Rome "Tor Vergata"}
\email[C. Durastanti]{ durastan@mat.uniroma2.it}
\author{Xiaohong Lan}
\address[X. Lan]{ University of Science and Technology of China (Hefei,
Anhui, China, 230026)}
\email[X. Lan]{ xhlan@ustc.edu.cn}
\date{Febbruary 2, 2013}

\begin{abstract}
This work develops the asymptotic properties (weak consistency and
Gaussianity), in the high-frequency limit, of approximate maximum likelihood
estimators for the spectral parameters of Gaussian and isotropic spherical
random fields. The procedure we used exploits the so-called mexican needlet
construction by Geller and Mayeli in \cite{gm2}. Furthermore, we propose a
plug-in procedure to optimize the precision of the estimators in terms of
asymptotic variance.
\end{abstract}

\maketitle

\section{Introduction}

The aim of this paper is to investigate the asymptotic behaviour of a
Whittle-like approximate maximum likelihood estimates of the spectral
parameters (e.g., the \emph{spectral index}) of isotropic Gaussian random
fields defined on the unit sphere $\mathbb{S}^{2}$. We employ a procedure
based on the so-called mexican needlet construction by Geller and Mayeli in 
\cite{gm2}. Furthermore, we develop a plug-in procedure aimed to merge and
to optimize these results with the achievements pursued in \cite{dlm}, \cite%
{dlm2}, see also \cite{durastanti}, where the asymptotic behaviour of
Whittle-like estimates were studied respectively in the harmonic and
standard needlet analysis frameworks.

Under the hypothesis of Gaussianity, fixing smoothness conditions on the
behaviour of the angular power spectrum, we pursue weak consistency and
central limit theorem allowing for feasible inference. From the technical
point of view, the asymptotic framework we use here is rather different from
the usual, being based on observations collected at higher and higher
frequencies on a fixed-domain (i. e. the unit sphere). In this sense, this
work can be related to the area of fixed-domain asymptotics (see for
instance \cite{anderes}, \cite{loh}); on the other hand, as for \cite{dlm}
and \cite{dlm2}, some of the techniques used here are close to those adopted
by \cite{Robinson} to analyze the asymptotic behaviour of the semiparametric
estimates of the long memory parameter for the covariance of stationary
processes. In terms of the angular power spectrum, we shall also focus on
semiparametric models where only the high-frequency/small-scale behaviour of
the random field is constrained. In particular, we consider both full-band
and narrow-band estimates, where the latter allow unbiased estimation under
more general assumption, by paying the price of a slower rate of convergence
if compared to the former.

This investigation, as many others regarding statistical inference on
spherical random fields, is strongly motivated by practical applications,
especially in cosmology and astrophysics (see for instance \cite{marpecbook}
and the references therein). For instance, as described in \cite{dode2004}
and \cite{cama}, satellite missions such as \emph{WMAP} and \emph{Planck}
are now providing huge datasets on Cosmic Microwave Background (CMB)
radiation, usually assumed to be a realization of an isotropic, Gaussian
spherical random field: the issues concerning parameter estimation have been
considered by many applied papers (see \cite{hamann}, \cite{Larson} for a
review), but in our knowledge, until now, rigorous asymptotic results are
still missing in literature. We however refer also to \cite{bkmpAoS}, \cite%
{fay08}, \cite{glm}, \cite{pbm06}, \cite{pietrobon1}, \cite{marpec2} for
further theoretical and applied results on angular power spectrum estimation
in nonparametric settings, and to \cite{kerkyphampic}, \cite{kim}, \cite%
{kimkoo}, \cite{kookim}, \cite{leonenko2}, \cite{ejslan} and \cite%
{marpecbook} for further results on statistical inference for spherical
random fields or wavelets applied to CMB radiation.

Another result we work out in this paper concerns the formulation of a
plug-in procedure which combines the application of the asymptotic results
here attained with those described in \cite{dlm} and \cite{dlm2}, where the
authors proved that weak consistency and central limit theorem can be
achieved respectively by standard Fourier and standard spherical needlet
analysis. In \cite{dlm}, the authors themselves have put in evidence that,
if the asymptotic achievements are better with respect to those obtain in
needlet framework in terms of precision of the estimates (e.g. their
asymptotic variance is smaller), in many practical circumstances the
implementation of spherical harmonics estimates may present some
difficulties, due to their lack of localization in real space. The presence
of unobserved regions on the sphere (common situation in the case of
Cosmological applications), can indeed make their implementation infeasible,
and spherical harmonics exclude the possibility of separate estimation on
different hemispheres, as considered for instance by \cite{bkmpBer}, \cite%
{pietrobon2}. In view of these issues, in \cite{dlm2}, the authors
investigated the Whittle-like procedures to a spherical wavelet framework,
in order to exploit the double-localization properties (in real and harmonic
space) of such constructions, at the cost of a smaller precision in term of
convergence in law of the estimates. They focussed their attention on
spherical needlets, second-generation wavelets on the sphere, introduced in
2006 by \cite{npw1} and \cite{npw2}, and very extensively exploited both in
the statistical literature and for astrophysical applications in the last
few years: for instance, their stochastic properties are developed in \cite%
{bkmpAoS}, \cite{bkmpBer}, \cite{bkmpAoSb} \cite{ejslan}, \cite{spalan} and 
\cite{mayeli}. More recently, needlets have been generalized in different
ways: we cite spin needlets (see \cite{gelmar}), and mixed needlets (cfr. 
\cite{gelmar2010}), which represent the natural generalization to the case
of spin fiber bundles, again developed in view of Cosmological applications
such as weak gravitational lensing and the polarization of the Cosmic
Microwave Background (CMB) radiation (see for instance \cite{bkmpAoS}, \cite%
{cama}, \cite{dmg}, \cite{fay08}, \cite{glm}, \cite{ghmkp}, \cite{mpbb08}, 
\cite{pietrobon1}, \cite{pietrobon2}, \cite{rudjord2}). On the other hand,
needlets have been generalized to an unbounded support in the frequency
domain by \cite{gm1}, \cite{gm2} and \cite{gm3}, the so-called Mexican
needlets. In this case, as we will describe in details below, even if the
support in frequency domain is unbounded, the form of the weight function,
depending on the scale parameter $p$, is such that for each wavelet there is
a small numbers of frequencies which give a contribution substantially far
from zero , while in the real domain the same weight function allows a
closer localization than the one related to standard spherical needlets. In
particular this double localization depends on the value of $p$ or, better,
on its distance from the spectral index, allowing these estimates to be more
efficient than the ones obtained with standard needlets. Our idea,
therefore, is to build a plug-in procedure on two steps, the first step
being to estimate approximately the value of the spectral index by standard
needlets and the second step providing a estimation with mexican needlets,
whereas the value of the scale parameter $p$ will allow a more efficent
estimator.

The plan of the paper is as follows: in Section \ref{sec: randomfields}, we
will recall some background material on mexican needlet analysis for
spherical isotropic random fields; in Section \ref{sec: needwhittle} we will
introduce and describe the Whittle-like minimum contrast estimators, while
in Section \ref{asympt} we shall establish the asymptotic results on these
estimators.\ In Section \ref{narrow} we present results on narrow band
estimates, while in Section\ \ref{plugin} we will describe the plug-in
procedure mentioned above. Finally, the appendix collects some analytical
and statistical auxiliary results.

\section{Random fields and mexican needlets\label{sec: randomfields}}

In this Section we will introduce the mexican needlet framework (for more
details, cfr. \cite{gm2}) and its application to the study of the isotropic,
Gaussian random fields on the sphere. First of all, consider the set of
spherical harmonics $\left\{ Y_{lm}:l\geq 0,m=-l,...,l\right\} $. As
well-known, it represents an orthonormal basis for the class of
square-integrable functions on the unit sphere space $L^{2}\left( \mathbb{S}%
^{2}\right) $: the spherical harmonics are defined as the eigenfunctions of
the spherical Laplacian $\Delta _{S^{2}}$ corresponding to eigenvalues $%
-l(l+1)$ (see, for more details and analytic expressions, \cite{adler} \cite%
{steinweiss}, \cite{VMK}, \cite{marpecbook} and, for extensions, \cite{leosa}%
, \cite{mal}). The mexican needlets are defined in \cite{gm2} as%
\begin{equation}
\psi _{jk;p}\left( x\right) :=\sqrt{\lambda _{jk}}\sum_{l\geq 1}f_{p}\left( 
\frac{l}{B^{j}}\right) \sum_{m=-l}^{l}\overline{Y}_{lm}\left( x\right)
Y_{lm}\left( \xi _{jk}\right) \text{ ,}  \label{mexneed}
\end{equation}%
where%
\begin{equation}
f_{p}\left( x\right) =x^{2p}\exp \left( -x^{2}\right) \text{ .}
\label{weightfunc}
\end{equation}%
Observe that $\left\{ \xi _{jk}\right\} $ is a set of cubature points on the
sphere, indexed by resolution level index $j$ and the cardinality of the
point over the fixed resolution level $k$, while $\lambda _{jk}>0$
corresponds to the weight associated to any $\xi _{jk}$. The scalar $N_{j}$
denotes the number of cubature points for a given level $j$ (cfr. \cite{npw1}%
, \cite{npw2}, see also e.g. \cite{gm2} and \cite{marpecbook}), chosen to
satisfy the following 
\begin{equation}
\lambda _{jk}\approx B^{-2j}\text{ },\text{ }N_{j}\approx B^{2j}\text{ ,}
\label{Njdef}
\end{equation}%
where by $a\approx b$, we mean that there exists $c_{1},c_{2}>0$ such that $%
c_{1}a\leq b\leq c_{2}a$. Below, we shall assume for notational simplicity,
as in \cite{dlm2},\ that there exists a positive constant $c_{B}$ such that $%
N_{j}=c_{B}B^{2j}$ for all resolution levels $j$. In practice, cubature
points and weights can be identified with those evaluated by common packages
such as HealPix (see for instance \cite{bkmpAoS}, \cite{GLESP}, \cite%
{HEALPIX}).

Considering $L_{l}(\left\langle x,y\right\rangle )=\sum_{m=-l}^{l}\overline{Y%
}_{lm}\left( x\right) Y_{lm}\left( y\right) $ as a projection operator, the
definition (\ref{mexneed}) corresponds to a weighted convolution with a
weight function (\ref{weightfunc}): mexican needlets can be considered as an
extension of the spherical standard needlets, proposed in \cite{npw1}, \cite%
{npw2}, see also \cite{bkmpAoSb}, \cite{dlm}, \cite{marpecbook}. The main
difference between these two kinds of wavelets concerns their dependence on
frequencies. In fact while standard needlets have a compact frequency
support (see again \cite{npw1}, \cite{npw2}), each mexican needlet is
defined on the whole frequency range. In \cite{gm2}, mexican needlets are
proved to form a nearly tight frame, differently from the standard needlets
which describe a tight frame and, as consequence, are characterized by an
exact reconstruction formula (see again \cite{npw1}).

Consider now a zero-mean, isotropic Gaussian random fields $T:\mathbb{S}%
^{2}\times \Omega \rightarrow \mathbb{R}$; it is a well known fact that for
every $g\in SO\left( 3\right) $ and $x\in \mathbb{S}^{2}$, a field $T\left(
\cdot \right) $ is isotropic if and only if%
\begin{equation*}
T\left( x\right) \overset{d}{=}T\left( gx\right) \text{ ,}
\end{equation*}%
where the equality holds in the sense of processes (see \cite{marpec2}, \cite%
{marpecbook}), and that (see e.g. \cite{marpecbook}) the following spectral
representation holds: 
\begin{equation*}
T\left( x\right) =\sum_{l\geq 0}\sum_{m=-l}^{l}a_{lm}Y_{lm}\left( x\right) 
\text{ , \ }a_{lm}=\int_{\mathbb{S}^{2}}T\left( x\right) \overline{Y}%
_{lm}\left( x\right) dx\text{ .}
\end{equation*}%
Note that this equality holds in both $L^{2}\left( \mathbb{S}^{2}\times
\Omega ,dx\otimes \mathbb{P}\right) $ and $L^{2}\left( \mathbb{P}\right) $
senses for every fixed $x\in \mathbb{S}^{2}$. For an isotropic Gaussian
field, the spherical harmonics coefficients $a_{lm}$ are Gaussian complex
random variables such that%
\begin{equation*}
\mathbb{E}\left( a_{lm}\right) =0\text{ , }\mathbb{E}\left( a_{lm}\overline{a%
}_{l_{1}m_{1}}\right) =\delta _{l}^{l_{1}}\delta _{m}^{m_{1}}C_{l}\text{ .}
\end{equation*}%
The angular power spectrum $\left\{ C_{l}\text{ , }l=1,2,3,...\right\} $
fully characterizes the dependence structure under Gaussianity. Properties
of the spherical harmonics coefficients under Gaussianity and isotropy are
discussed for instance by \cite{bm}, \cite{marpecbook}; here we recall that 
\begin{equation*}
\sum_{m=-l}^{l}\left\vert a_{lm}\right\vert ^{2}\sim C_{l}\times \chi
_{2l+1}^{2}\text{ .}
\end{equation*}%
Hence, given a realization of the random field, an estimator of the angular
power spectrum can be defined as:%
\begin{equation*}
\widehat{C}_{l}=\frac{1}{2l+1}\sum_{m=-l}^{l}\left\vert a_{lm}\right\vert
^{2}\text{ ,}
\end{equation*}%
the empirical angular power spectrum. It is immediately observed that%
\begin{equation}
\mathbb{E}\left( \widehat{C}_{l}\right) =\frac{1}{2l+1}%
\sum_{m=-l}^{l}C_{l}=C_{l}\text{ , }Var\left( \frac{\widehat{C}_{l}}{C_{l}}%
\right) =\frac{2}{2l+1}\rightarrow 0\text{ for }l\rightarrow +\infty \text{ .%
}  \label{powest2}
\end{equation}

As in \cite{dlm2}, we introduce the following regularity condition on the
angular power spectrum:

\begin{condition}[Regularity]
\label{REGULNEED0}The random field $T(x)$ is Gaussian and isotropic with
angular power spectrum $C_{l}$ so that for all $B>1$, there exist $\alpha
_{0}>2$, $c_{0}>0$ such that:%
\begin{equation}
C_{l}=l^{-\alpha _{0}}G\left( l\right) >0,\text{ for all }l\in \mathbb{N}%
\text{ ,}  \label{Cl-reg1}
\end{equation}%
where $c_{0}^{-1}\leq G\left( l\right) \leq c_{0}$\ for all $l\in N$ , and
for every $r\in \mathbb{N}$, there exists $c_{r}>0$ such that:%
\begin{equation*}
\left\vert \frac{d^{r}}{du^{r}}G\left( u\right) \right\vert \leq c_{r}u^{-r}%
\text{ , }\in \left( 0,+\infty \right) \text{ .}
\end{equation*}
\end{condition}

This assumption is fulfilled by popular physical models, for instance in a
CMB framework the \emph{Sachs-Wolfe} power spectrum, which is the leading
model for fluctuations of the primordial gravitational potential, takes the
form (\ref{Cl-reg1}), see for instance \cite{dode2004}.

First of all, we stress that Condition \ref{REGULNEED0} implies the
following Condition \ref{condspalan}, given in \cite{spalan}.

\begin{condition}
\label{condspalan}Condition \ref{REGULNEED0} holds and, moreover, there
exist $\alpha _{0}>2$ and a sequence of functions $\left\{ g_{j}\left( \cdot
\right) \right\} _{j=1,2,...}$ such that:%
\begin{equation}
C_{l}=l^{-\alpha _{0}}g_{j}\left( \frac{l}{B^{j}}\right) >0\text{, for all }%
B^{j-1}<l<B^{j+1}\text{, }j=1,2...  \label{Cl-reg2}
\end{equation}%
where $c_{0}^{-1}\leq g_{j}\leq c_{0}$\ for all $j\in N$ , and for every $%
r=0,...,Q$, $Q\in \mathbb{N}$, there exists $c_{r}>0$ such that:%
\begin{equation*}
\sup_{j}\sup_{B^{j-1}<u<B^{j+1}}\left\vert \frac{d^{r}}{du^{r}}g_{j}\left(
u\right) \right\vert \leq c_{r}.
\end{equation*}
\end{condition}

As an example, consider%
\begin{equation*}
C_{l}=l^{-\alpha }\frac{P\left( l\right) }{Q\left( l\right) }\text{ ,}
\end{equation*}%
where $P\left( l\right) =\sum_{i=1}^{p}c_{p,i}l^{i}$ and $Q\left( l\right)
=\sum_{i=1}^{q}c_{q,i}l^{i}$ are positive polynomials of degree $p$ and $q$
respectively, so that $\alpha _{0}=\alpha -p+q>2$., so that%
\begin{eqnarray*}
C_{l} &=&l^{-\alpha +p-q}\frac{c_{p,p}}{c_{q,q}}\frac{1+\frac{c_{p,p-1}}{%
c_{p,p}}\frac{1}{l}+\frac{c_{p,p-2}}{c_{p,p}}\frac{1}{l^{2}}+...}{1+\frac{%
c_{q,q-1}}{c_{q,q}}\frac{1}{l}+\frac{c_{q,q-2}}{c_{q,q}}\frac{1}{l^{2}}+...}
\\
&=&l^{-\alpha +p-q}\frac{c_{p,p}}{c_{q,q}}\frac{1+\frac{1}{Bj}\frac{c_{p,p-1}%
}{c_{p,p}}\frac{B^{j}}{l}+\frac{1}{B^{2j}}\frac{c_{p,p-2}}{c_{p,p}}\left( 
\frac{B^{j}}{l}\right) ^{2}+......}{1+\frac{1}{Bj}\frac{c_{q,q-1}}{c_{q,q}}%
\frac{1}{l}\frac{B^{j}}{l}+\frac{1}{B^{2j}}\frac{c_{q,q-2}}{c_{q,q}}\frac{1}{%
l^{2}}\left( \frac{B^{j}}{l}\right) ^{2}+...} \\
&=&l^{-\alpha _{0}}g_{j}\left( \frac{l}{B^{j}}\right) \text{ .}
\end{eqnarray*}

Condition \ref{REGULNEED0} will be necessary to prove needlet coefficients (%
\ref{Mbeta}) to be asymptotically uncorrelated\ (see \cite{spalan}, \cite%
{mayeli}); as we shall show, Condition \ref{REGULNEED0} is sufficient to
establish consistency for estimator we are going to define but we will
consider two further nested assumptions, \ref{REGULNEED2} (which implies and
is implied by \ref{REGULNEED0}), to obtain asymptotic Gaussianity, and \ref%
{REGULNEED3} (which implies \ref{REGULNEED2}) to provide a centered limiting
distribution, see also \cite{dlm}, \cite{dlm2}.

\begin{condition}
\label{REGULNEED2} Condition \ref{REGULNEED0} holds and moreover%
\begin{equation*}
G\left( l\right) =G_{0}\left( 1+\kappa l^{-1}+O\left( l^{-2}\right) \right) 
\text{ .}
\end{equation*}
\end{condition}

\begin{condition}
\label{REGULNEED3} Condition \ref{REGULNEED0} holds and moreover 
\begin{equation*}
G\left( l\right) =G_{0}\left( 1+o\left( l^{-1}\right) \right) \text{ .}
\end{equation*}
\end{condition}

For any given $\,j$, $k$, $p$, we define the needlet coefficients as: 
\begin{eqnarray}
\beta _{jk;p} &:&=\int_{S^{2}}T\left( x\right) \overline{\psi }_{jk;p}\left(
x\right) dx  \notag \\
&=&\sqrt{\lambda _{jk}}\underset{l\geq 1}{\sum }f_{p}\left( \frac{l}{B^{j}}%
\right) \overset{l}{\underset{m=-l}{\sum }}a_{lm}Y_{lm}\left( \xi
_{jk}\right) \text{ ,}  \label{Mbeta}
\end{eqnarray}%
so that%
\begin{equation*}
\mathbb{E}\left( \beta _{jk;p}\right) =\sqrt{\lambda _{jk}}\underset{l\geq 1}%
{\sum }f_{p}\left( \frac{l}{B^{j}}\right) \overset{l}{\underset{m=-l}{\sum }}%
\mathbb{E}\left( a_{lm}\right) Y_{lm}\left( \xi _{jk}\right) =0\text{ .}
\end{equation*}%
Under Condition \ref{condspalan}, the following result is given in \cite%
{spalan} and \cite{mayeli}.

\begin{lemma}
\label{Mcorr}If $0<4p+2-\alpha _{0}\leq Q$, then under Condition \ref%
{condspalan}, there exists a constant $C_{Q}>0,$ such that%
\begin{equation*}
Corr\left( \beta _{jk;p},\beta _{j^{\prime }k^{\prime };p}\right) \leq \frac{%
C_{Q}}{\left[ 1+B^{\left( \left( j+j^{\prime }\right) /2-\log _{B}\left[
\left( j+j^{\prime }\right) /2\right] \right) }d\left( \xi _{jk},\xi
_{j^{\prime }k^{\prime }}\right) \right] ^{\left( 4p+2-\alpha _{0}\right) }}.
\end{equation*}
\end{lemma}

Assume now that from the observations over the random field, we are able to
build the following set of quantities 
\begin{equation*}
\widehat{\Lambda }_{j;p}:=\sum_{l\geq 1}f_{p}^{2}\left( \frac{l}{B^{j}}%
\right) \left( 2l+1\right) \widehat{C}_{l}\text{ }\simeq
\sum_{k=1}^{N_{j}}\beta _{jk;p}^{2}\text{ }\ \text{for each }j\in \left[
J_{0},J_{L}\right] \text{ ,}
\end{equation*}%
where the last approximation is motivated by the nearly tight frame
property, as in \cite{mayeli}.

The next result describes the asymptotic behaviour of the
variance-covariance matrix of $\widehat{\Lambda }_{j;p}$ in terms of $j$.

\begin{lemma}
\label{mexicangavarini}If Condition \ref{REGULNEED0} holds with $%
0<4p+2-\alpha _{0}\leq Q$, fixed $\Delta j\in \mathbb{Z}$, we have%
\begin{equation*}
\lim_{j\rightarrow \infty }\frac{1}{B^{2\left( 1-\alpha _{0}\right) j}}%
Var\left( \widehat{\Lambda }_{j;p}\right) =\frac{2G_{0}^{2}}{4^{4p+\left(
1-\alpha _{0}\right) }}\Gamma \left( 4p+1-\alpha _{0}\right) \text{ ;}
\end{equation*}%
\begin{equation*}
\lim_{j\rightarrow \infty }\frac{1}{B^{2\left( 1-\alpha _{0}\right) j}}%
Cov\left( \Lambda _{j;p},\Lambda _{j+\Delta j;p}\right) =2G_{0}^{2}\frac{%
\tau _{B}\left( \Delta j\right) }{4^{4p+\left( 1-\alpha _{0}\right) }}\Gamma
\left( 4p+1-\alpha _{0}\right) \text{ ,}
\end{equation*}%
where%
\begin{equation}
\tau _{p}\left( \Delta j\right) :=B^{\Delta j\left( 1-\alpha _{0}\right)
}\cosh \left( \Delta j\log B\right) ^{-\left( 4p-\alpha _{0}+1\right) }\text{
.}  \label{tau_def}
\end{equation}
\end{lemma}

\begin{proof}
Simple calculations lead to:%
\begin{eqnarray*}
Var\left( \Lambda _{j;p}\right) &=&Var\left( \sum_{k=1}^{N_{j}}\beta
_{jk;p}^{2}\right) =\sum_{l\geq 1}f_{p}^{4}\left( \frac{l}{B^{j}}\right)
\left( 2l+1\right) ^{2}Var\left( \widehat{C}_{l}\right) \\
&=&2\sum_{l\geq 1}f_{p}^{4}\left( \frac{l}{B^{j}}\right) \left( 2l+1\right)
C_{l}^{2}\text{ ,}
\end{eqnarray*}%
while, for $\Delta j\in \mathbb{Z}$, 
\begin{equation*}
Cov\left( \Lambda _{j;p},\Lambda _{j+\Delta j;p}\right) =Cov\left(
\sum_{k_{1}=1}^{N_{j}}\beta _{jk_{1};p}^{2},\sum_{k_{2}=1}^{N_{j+\Delta
j}}\beta _{j+\Delta jk_{2};p}^{2}\right)
\end{equation*}%
\begin{eqnarray*}
&=&Cov\left( \sum_{l_{1}\geq 1}f_{p}^{2}\left( \frac{l_{1}}{B^{j}}\right)
\left( 2l_{1}+1\right) \widehat{C}_{l_{1}},\sum_{l_{2}\geq 1}\left( \frac{%
l_{2}}{B^{j+\Delta j}}\right) \left( 2l_{2}+1\right) \widehat{C}%
_{l_{2}}\right) \\
&=&\sum_{l\geq 1}f_{p}^{2}\left( \frac{l}{B^{j}}\right) f_{p}^{2}\left( 
\frac{l}{B^{j+\Delta j}}\right) \left( 2l+1\right) ^{2}Var\left( \widehat{C}%
_{l}\right) \\
&=&2\sum_{l\geq 1}f_{p}^{2}\left( \frac{l}{B^{j}}\right) f_{p}^{2}\left( 
\frac{l}{B^{j+\Delta j}}\right) \left( 2l+1\right) C_{l}^{2}\text{ .}
\end{eqnarray*}%
Under Condition \ref{REGULNEED0},\ by applying Lemma \ref{lemmasums}, in
view of the equation (\ref{sumvar}) with $a=4$ and $n=1-2\alpha _{0}$, we
have: 
\begin{eqnarray*}
Var\left( \Lambda _{j;p}\right) &=&4G_{0}^{2}\sum_{l\geq 1}f_{p}^{4}\left( 
\frac{l}{B^{j}}\right) \left( l^{1-2\alpha _{0}}+o_{l}\left( l^{1-2\alpha
_{0}}\right) \right) \\
&=&2G_{0}^{2}\frac{B^{2\left( 1-\alpha _{0}\right) j}}{4^{4p+\left( 1-\alpha
_{0}\right) }}\Gamma \left( 4p+1-\alpha _{0}\right) +o_{j}\left( B^{2\left(
1-\alpha _{0}\right) j}\right) \text{ ,}
\end{eqnarray*}%
while, for the equation (\ref{sumcov}) with $a_{1}=a_{2}=2$, $n=1-2\alpha
_{0}$ and $\tau _{p}\left( \Delta j\right) =\tau _{p,2,2}\left( \Delta
j\right) $, we obtain:%
\begin{equation*}
Cov\left( \Lambda _{j;p},\Lambda _{j+\Delta j;p}\right)
\end{equation*}%
\begin{eqnarray*}
&=&4G_{0}^{2}\sum_{l\geq 1}f_{p}^{2}\left( \frac{l}{B^{j}}\right)
f_{p}^{2}\left( \frac{l}{B^{j+\Delta j}}\right) l^{1-2\alpha
_{0}}+o_{l}\left( 1\right) \\
&=&2G_{0}^{2}\frac{\tau _{B}\left( \Delta j\right) }{4^{4p+\left( 1-\alpha
_{0}\right) }}B^{2\left( 1-\alpha _{0}\right) j}\Gamma \left( 4p+1-\alpha
_{0}\right) +o\left( B^{2\left( 1-\alpha _{0}\right) j}\right) \text{ ,}
\end{eqnarray*}%
as claimed.
\end{proof}

\section{Mexican Needlet Whittle-like approximation to likelihood function 
\label{sec: needwhittle}}

In this Section, our aim is to define a mexican needlet Whittle-like
approximation to the log-likelihood function of isotropic and Gaussian
random fields on the unit sphere under Condition \ref{REGULNEED0} and to
develop the corresponding estimators. We will follow a strategy analogue to
the one used by \cite{dlm2}, (see also \cite{dlm} and \cite{Robinson}). We
let 
\begin{equation*}
\overrightarrow{\mathbf{\beta }}_{j;p}=\left( \beta _{j1;p},\beta
_{j2;p},...,\beta _{jN_{j};p}\right)
\end{equation*}%
where $\beta _{jk;p}$ is defined as in (\ref{Mbeta}). Again, under the
hypothesis of isotropy and Gaussianity for $T$,$\ $we have 
\begin{equation*}
\overrightarrow{\mathbf{\beta }}_{j;p}\sim N\left( 0,\Gamma \right) \text{ ,}
\end{equation*}%
where 
\begin{equation*}
\Gamma =\left[ Cov\left( \beta _{jk;p},\beta _{jk^{\prime };p}\right) \right]
_{k,k^{\prime }}=\frac{1}{N_{j}}\left( \underset{l\geq 1}{\sum }%
f_{p}^{2}\left( \frac{l}{B^{j}}\right) \left( 2l+1\right) C_{l}\right)
I_{N_{j}}\ \text{,}
\end{equation*}%
in view of (\ref{Mbeta}) and Lemma \ref{Mcorr} (see also \cite{dlm}, \cite%
{dlm2}). \ The likelihood function is then defined as 
\begin{equation*}
\mathcal{L}\left( \vartheta ;\overrightarrow{\mathbf{\beta }}_{j;p}\right)
=\left( 2\pi \right) ^{-N_{j}}\left( \det \Gamma \right) ^{-1/2}\exp \left\{
-\frac{1}{2}\overrightarrow{\mathbf{\beta }}_{j;p}^{T}\Gamma ^{-1}%
\overrightarrow{\overline{\beta }}_{j;p}\right\} \text{ .}
\end{equation*}

Let

\begin{equation*}
K_{j}^{M}\left( \alpha \right) :=\frac{1}{N_{j}}\underset{l\geq 1}{\sum }%
f_{p}^{2}\left( \frac{l}{B^{j}}\right) \left( 2l+1\right) l^{-\alpha }\ .
\end{equation*}%
Under Condition \ref{REGULNEED0}, we have: 
\begin{equation*}
\mathcal{L}\left( \alpha ,G;\overrightarrow{\mathbf{\beta }}_{j;p}\right)
=\left( 2\pi \right) ^{-N_{j}}\left( GK_{j}^{M}\left( \alpha \right) \right)
^{-N_{j}/2}\exp \left\{ -\frac{1}{2}\frac{\sum_{k}\beta _{jk;p}^{2}}{%
GK_{j}^{M}\left( \alpha \right) }\right\} \text{ ,}
\end{equation*}%
and the corresponding approximate log-likelihood is%
\begin{equation*}
-2\log \mathcal{L}\left( \alpha ,G;\overrightarrow{\mathbf{\beta }}%
_{j;p}\right) =\sum_{k}\left\{ \frac{\beta _{jk;p}^{2}}{GK_{j}^{M}\left(
\alpha \right) }-\log \left( \frac{\beta _{jk;p}^{2}}{GK_{j}^{M}\left(
\alpha \right) }\right) \right\} \text{ ,}
\end{equation*}%
up to an additive constant.

By summing with respect to $j$, we obtain.%
\begin{equation*}
R_{J_{0},J_{L}}^{M}\left( \alpha ,G\right) :=\left(
\sum_{j=J_{0}}^{J_{L}}N_{j}\right) ^{-1}\sum_{j=J_{0}}^{J_{L}}-2\log 
\mathcal{L}\left( \alpha ,G;\mathbf{\beta }_{j;p}\right) \text{,}
\end{equation*}%
where the choice for $J_{0}$, $J_{L}$ will be discussed later. Hence we
define (cfr. \cite{dlm} and \cite{dlm2}) 
\begin{equation*}
\left( \widehat{\alpha }_{J_{0},J_{L}}^{M},\widehat{G}_{J_{0},J_{L}}^{M}%
\right) =\arg \min_{\left( \alpha ,G\right) \in \Theta
}R_{J_{0},J_{L}}^{M}\left( \alpha ,G\right) \text{ ,}
\end{equation*}%
where $\Theta =[2,+\infty )\times \left( 0,+\infty \right) .$ Computing the
derivative of $R_{J_{0},J_{L}}^{M}$ with respect to $G$ and setting it equal
to zero, we have 
\begin{equation*}
0=\frac{\partial }{\partial G}R_{J_{0},J_{L}}^{M}\left( \alpha ,G\right) =%
\frac{1}{\sum_{j=J_{0}}^{J_{L}}N_{j}}\sum_{j=J_{0}}^{J_{L}}\left[ -\frac{%
\sum_{k}\beta _{jk;p}^{2}}{G^{2}K_{j}^{M}\left( \alpha \right) }+\frac{N_{j}%
}{G}\right] \text{ ,}
\end{equation*}%
whence%
\begin{equation*}
\widehat{G}_{J_{0},J_{L}}^{M}=\widehat{G}_{J_{0},J_{L}}^{M}\left( \alpha
\right) =\frac{1}{\sum_{j=J_{0}}^{J_{L}}N_{j}}\sum_{j=J_{0}}^{J_{L}}\frac{%
\sum_{k}\beta _{jk;p}^{2}}{K_{j}^{M}\left( \alpha \right) }=\frac{1}{%
\sum_{j=J_{0}}^{J_{L}}N_{j}}\sum_{j=J_{0}}^{J_{L}}\frac{\Lambda _{j;p}}{%
K_{j}^{M}\left( \alpha \right) }.
\end{equation*}%
Since%
\begin{equation*}
\frac{\partial ^{2}}{\partial G^{2}}R_{J_{0},J_{L}}^{M}\left( \alpha
,G\right) \left\vert _{G=\widehat{G}_{J_{0},J_{L}}^{M}\left( \alpha \right)
}\right.
\end{equation*}%
\begin{equation*}
=\frac{1}{\sum_{j=J_{0}}^{J_{L}}N_{j}}\sum_{j=J_{0}}^{J_{L}}\frac{1}{G^{2}}%
\left( \frac{2\Lambda _{j;p}}{GK_{j}^{M}\left( \alpha \right) }-N_{j}\right)
\left\vert _{G=\widehat{G}_{J_{0},J_{L}}^{M}\left( \alpha \right) }\right. =%
\frac{1}{\left( \widehat{G}_{J_{0},J_{L}}^{M}\left( \alpha \right) \right)
^{2}}>0\text{ ,}
\end{equation*}%
and $R_{J_{0},J_{L}}^{M}\left( \alpha ,G\right) \rightarrow +\infty $ as $%
G\rightarrow 0\ $or$\ \infty ,$ the second derivative test yields that $%
R_{J_{0},J_{L}}^{M}\left( \alpha ,G\right) $ has a unique minimum over the
domain on $\widehat{G}_{J_{0},J_{L}}^{M}\left( \alpha \right) .$ Therefore,
we can define 
\begin{eqnarray*}
R_{J_{0},J_{L}}^{M}\left( \alpha \right) &:&=R_{J_{0},J_{L}}^{M}\left(
\alpha ,\widehat{G}_{J_{0},J_{L}}^{M}\left( \alpha \right) \right) \\
&=&1+\log \widehat{G}_{J_{0},J_{L}}^{M}\left( \alpha \right) -\frac{1}{%
\sum_{j=J_{0}}^{J_{L}}N_{j}}\sum_{j=J_{0}}^{J_{L}}\sum_{k}\log \frac{\beta
_{jk;p}^{2}}{K_{j}^{M}\left( \alpha \right) }\text{ .}
\end{eqnarray*}

\begin{remark}
In this formula it is necessary to fix explicitly the values of $L$, $J_{0}$
and $J_{L}$. Let us fix $L$ as the highest multipole level available from
the dataset. Given $L$, as stressed above, differently from the standard
needlet case (see for instance \cite{npw1}, \cite{npw2}), in the mexican
needlet case the weight function does not have a compact support. Therefore,
for computational reasons, we must find a criterion to choose the suitable
extrema for the sums over $j$ involved. Considering (see again \cite{gm2})
the behaviour of $f_{p}\left( \cdot \right) $, we can fix thresholds $%
\varepsilon _{B,1}\left( L\right) $, $\varepsilon _{B,2}\left( L\right) $,
such that:%
\begin{equation*}
J_{0}=\max \left\{ j\in \mathbb{Z}:f_{p}\left( \frac{1}{B^{j+1}}\right)
>\varepsilon _{B,1}\left( L\right) f_{p}\left( \frac{1}{B^{j}}\right)
\right\} \text{ ,}
\end{equation*}%
\begin{equation*}
J_{L}=\min \left\{ j\in \mathbb{Z}:f_{p}\left( \frac{L}{B^{j}}\right)
<\varepsilon _{B,2}\left( L\right) f_{p}\left( \frac{L}{B^{j-1}}\right)
\right\} \text{ .}
\end{equation*}%
If, for instance, we choose,%
\begin{equation*}
\varepsilon _{B,1}\left( L\right) =\frac{1}{B^{2p}}\exp \left( \frac{B-1}{%
B^{2}}\right) \text{, }\varepsilon _{B,2}\left( L\right) =\frac{1}{B^{2p}}%
\exp \left( B^{2}\left( B^{2}-1\right) \right)
\end{equation*}%
we find $B^{J_{0}}=B$, $B^{J_{L}}=L/B$, similarly to the classical needlet
case as described in \cite{dlm2}.
\end{remark}

\section{Asymptotic Properties\label{asympt}}

In this Section, we prove weak consistency for the estimators $\widehat{%
\alpha }_{J_{0},J_{L}}^{M}$ and $\widehat{G}_{J_{0},J_{L}}^{M}$, and for the
former also asymptotic Gaussianity. We begin with some definitions: let 
\begin{equation*}
G_{J_{0};J_{L}}^{M}\left( \alpha \right) =\frac{1}{%
\sum_{j=J_{0}}^{J_{L}}N_{j}}\sum_{j=J_{0}}^{J_{L}}N_{j}\frac{%
G_{0}K_{j}^{M}\left( \alpha _{0}\right) }{K_{j}^{M}\left( \alpha \right) }%
\text{ .}
\end{equation*}%
Computing the first and second order derivatives of $G_{J_{0};J_{L}}^{M}%
\left( \alpha \right) $, indexed by $n$, we obtain%
\begin{eqnarray*}
G_{J_{0};J_{L},n}^{M}\left( \alpha \right) &:&=\frac{d^{n}}{d\alpha ^{n}}%
G_{J_{0};J_{L}}^{M}\left( \alpha \right) \\
&=&\frac{G_{0}}{\sum_{j=J_{0}}^{J_{L}}N_{j}}\sum_{j=J_{0}}^{J_{L}}N_{j}\frac{%
K_{j}^{M}\left( \alpha _{0}\right) }{K_{j}^{M}\left( \alpha \right) }%
U_{j;n}\left( \alpha \right) \text{ ,}
\end{eqnarray*}%
where (see Proposition \ref{propsumKj}) in the Appendix, we have 
\begin{equation}
U_{j;1}\left( \alpha \right) =\left( -\frac{K_{j,1}^{M}\left( \alpha \right) 
}{K_{j}^{M}\left( \alpha \right) }\right) =\left( j\log B+\frac{%
I_{p,1}\left( \alpha \right) }{I_{p,0}\left( \alpha \right) }%
+o_{j}(1)\right) \text{ ,}  \label{UM1}
\end{equation}%
\begin{equation*}
U_{j;2}\left( \alpha \right) =2\left( \frac{K_{j,1}^{M}\left( \alpha \right) 
}{K_{j}^{M}\left( \alpha \right) }\right) ^{2}-\frac{K_{j,2}^{M}\left(
\alpha \right) }{K_{j}^{M}\left( \alpha \right) }
\end{equation*}%
\begin{equation}
=j^{2}\log ^{2}B+2j\log B\frac{I_{p,1}\left( \alpha \right) }{I_{p,0}\left(
\alpha \right) }+2\left( \frac{I_{p,1}\left( \alpha \right) }{I_{p,0}\left(
\alpha \right) }\right) ^{2}-\frac{I_{p,2}\left( \alpha \right) }{%
I_{p,0}\left( \alpha \right) }+o_{j}\left( 1\right) \text{ ,}  \label{UM2}
\end{equation}%
Furthermore, we fix 
\begin{equation*}
U_{j;0}\left( \alpha \right) =1,G_{J_{0};J_{L},0}^{M}\left( \alpha \right)
=G_{J_{0};J_{L}}^{M}\left( \alpha \right) \text{ ,}
\end{equation*}
(since now, we will use either $G_{J_{0};J_{L},0}^{M}\left( \alpha \right) $
or $G_{J_{0};J_{L}}^{M}\left( \alpha \right) $). Recalling that $%
N_{j}=C_{B}B^{2j}.$ Thus by (\ref{MKalpha}), we have for $s=0,1,2,$ 
\begin{eqnarray*}
G_{J_{0};J_{L},s}^{M}\left( \alpha \right) &=&\frac{G_{0}}{%
\sum_{j=J_{0}}^{J_{L}}N_{j}}\sum_{j=J_{0}}^{J_{L}}N_{j}\frac{K_{j}^{M}\left(
\alpha _{0}\right) }{K_{j}^{M}\left( \alpha \right) }U_{j;s}\left( \alpha
\right) \\
&=&G_{0}\frac{\left( p+1\right) ^{-\frac{\alpha -\alpha _{0}}{2}}}{%
\sum_{j=J_{0}}^{J_{L}}B^{2j}}\sum_{j=J_{0}}^{J_{L}}B^{\left( 2+\alpha
-\alpha _{0}\right) j}U_{j;s}\left( \alpha \right) .
\end{eqnarray*}

The next auxiliary result is as follows:

\begin{lemma}
\label{Gbehave}Assume Condition \ref{REGULNEED0} holds with $0<4p+2-\alpha
_{0}\leq Q$. We have that 
\begin{equation*}
\lim_{J_{L}\rightarrow \infty }\mathbb{E}\left( \widehat{G}%
_{J_{0},J_{L}}^{M}\left( \alpha _{0}\right) \right) \rightarrow G_{0}\text{,}
\end{equation*}%
\begin{equation*}
\lim \frac{1}{B^{2J_{L}}}Var\left( \frac{\widehat{G}_{J_{0},J_{L}}^{M}\left(
\alpha _{0}\right) }{G_{0}}\right) =\frac{B^{2}-1}{B^{2}}\sigma
_{0}^{2}\left( 1+\widetilde{\tau }_{0}\right) \text{ ,}
\end{equation*}%
where 
\begin{equation*}
\sigma _{0}^{2}:=\sigma _{0}^{2}\left( p,\alpha _{0}\right) =\frac{2}{%
2^{4p-\alpha _{0}}}\frac{\Gamma \left( 4p+1-\alpha _{0}\right) }{\Gamma
^{2}\left( 2p-\frac{\alpha _{0}}{2}+1\right) }\text{ ,}
\end{equation*}%
and $\widetilde{\tau }_{0}$ is as defined in Lemma \ref{Lemmasumtau}.
\end{lemma}

\begin{proof}
We have 
\begin{eqnarray*}
\mathbb{E}\left( \widehat{G}_{J_{0},J_{L},}^{M}\left( \alpha _{0}\right)
\right) &=&\frac{1}{\sum_{j=J_{0}}^{J_{L}}N_{j}}\sum_{j=J_{0}}^{J_{L}}\frac{%
\mathbb{E}\left( \widehat{\Lambda }_{j}\right) }{K_{j}^{M}\left( \alpha
\right) } \\
&=&\frac{G_{0}}{\sum_{j=J_{0}}^{J_{L}}N_{j}}\sum_{j=J_{0}}^{J_{L}}\frac{%
\sum_{l}f_{p}^{2}\left( \frac{l}{B^{j}}\right) \left( 2l+1\right) l^{-\alpha
_{0}}\left( 1+O\left( l^{-1}\right) \right) }{K_{j}^{M}\left( \alpha \right) 
} \\
&=&G_{0}+o_{J_{L}}\left( 1\right) \text{ .}
\end{eqnarray*}%
On the other hand, we prove that%
\begin{equation}
Cov\left( \frac{\Lambda _{j;p}}{G_{0}K_{j}^{M}\left( \alpha \right) },\frac{%
\Lambda _{j+\Delta j;p}}{G_{0}K_{j+\Delta j}^{M}\left( \alpha \right) }%
\right) =c_{B}^{2}\sigma _{0}^{2}B^{2j}B^{\alpha _{0}\Delta j}\tau
_{B}\left( \Delta j\right) \text{ .}  \label{covnormalized}
\end{equation}%
We can indeed observe from Theorem \ref{mexicangavarini} that%
\begin{equation*}
Cov\left( \frac{\widehat{\Lambda }_{j}}{G_{0}K_{j}^{M}\left( \alpha \right) }%
,\frac{\widehat{\Lambda }_{j+\Delta j}}{G_{0}K_{j+\Delta j}^{M}\left( \alpha
\right) }\right)
\end{equation*}%
\begin{eqnarray}
&=&\frac{B^{\alpha \Delta j}}{G_{0}^{2}I_{p,0}^{2}\left( \alpha \right)
B^{-2\alpha j}}Cov\left( \widehat{\Lambda }_{j},\widehat{\Lambda }_{j+\Delta
j}\right)  \notag \\
&=&\frac{B^{\alpha \Delta j}}{I_{p,0}^{2}\left( \alpha \right) }\frac{%
2\Gamma \left( 4p+1-\alpha _{0}\right) }{4^{4p+1-\alpha _{0}}}\tau
_{B}\left( \Delta j\right) B^{2\left( 1+\alpha -\alpha _{0}\right) j}  \notag
\\
&=&\frac{2c_{B}^{2}}{2^{4p-2\alpha _{0}+\alpha }}\frac{\Gamma \left(
4p+1-\alpha _{0}\right) }{\Gamma ^{2}\left( 2p-\frac{\alpha }{2}+1\right) }%
\tau _{B}\left( \Delta j\right) B^{\alpha \Delta j}B^{2\left( 1+\alpha
-\alpha _{0}\right) j}\text{ .}  \label{covcov}
\end{eqnarray}%
Hence%
\begin{equation*}
Var\left( \frac{\widehat{G}_{J_{0},J_{L}}^{M}\left( \alpha \right) }{G_{0}}%
\right)
\end{equation*}%
\begin{eqnarray*}
&=&\frac{1}{\left( \sum_{j=J_{0}}^{J_{L}}N_{j}\right) ^{2}}Cov\left(
\sum_{j=J_{0}}^{J_{L}}\frac{\Lambda _{j;p}}{G_{0}K_{j}^{M}\left( \alpha
\right) },\sum_{\Delta j=J_{0}-j}^{J_{L}-j}\frac{\Lambda _{j+\Delta j;p}}{%
G_{0}K_{j+\Delta j}^{M}\left( \alpha \right) }\right) \\
&=&\frac{1}{\left( \sum_{j=J_{0}}^{J_{L}}N_{j}\right) ^{2}}%
\sum_{j=J_{0}}^{J_{L}}\sum_{\Delta j=J_{0}-j}^{J_{L}-j}Cov\left( \frac{%
\sum_{k_{1}}\beta _{jk_{1};p}^{2}}{G_{0}K_{j}^{M}\left( \alpha \right) },%
\frac{\sum_{k_{2}}\beta _{j+\Delta j,k;p}^{2}}{G_{0}K_{j+\Delta j}^{M}\left(
\alpha \right) }\right) \\
&=&\frac{1}{\left( \sum_{j=J_{0}}^{J_{L}}B^{2j}\right) ^{2}}\frac{1}{%
4^{2p-\alpha _{0}+\frac{\alpha }{2}}}\frac{\Gamma \left( 4p+1-\alpha
_{0}\right) }{\Gamma ^{2}\left( 2p-\frac{\alpha }{2}+1\right) }%
\sum_{j=J_{0}}^{J_{L}}B^{2\left( 1+\alpha -\alpha _{0}\right) j}\sum_{\Delta
j=J_{0}-j}^{J_{L}-j}\tau _{B}\left( \Delta j\right) B^{\alpha \Delta j}\text{
.}
\end{eqnarray*}%
Following Lemmas \ref{Lemmasumtau} and \ref{prop27}, and computing in $%
\alpha =\alpha _{0}$, we have%
\begin{eqnarray*}
Var\left( \frac{\widehat{G}_{J_{0},J_{L}}^{M}\left( \alpha _{0}\right) }{%
G_{0}}\right) &=&\frac{2\left( 1+\widetilde{\tau }_{0}\right) }{2^{4p-\alpha
_{0}}}\frac{\Gamma \left( 4p+1-\alpha _{0}\right) }{\Gamma ^{2}\left( 2p-%
\frac{\alpha _{0}}{2}+1\right) }\left( \sum_{j=J_{0}}^{J_{L}}B^{2j}\right)
^{-1} \\
&=&\frac{B^{2}-1}{B^{2}}\sigma _{0}^{2}\left( 1+\widetilde{\tau }_{0}\right)
B^{-2J_{L}}+o\left( B^{-2J_{L}}\right) \text{ .}
\end{eqnarray*}
\end{proof}

\begin{lemma}
\label{unifconvergence}Under Condition \ref{REGULNEED0}, we have for $%
s=0,1,2 $:%
\begin{equation*}
\sup \left\vert \frac{\widehat{G}_{J_{0},J_{L};s}^{M}\left( \alpha \right) }{%
G_{J_{0},J_{L};s}^{M}\left( \alpha \right) }\right\vert \rightarrow _{p}0%
\text{ .}
\end{equation*}
\end{lemma}

\begin{proof}
Under Condition \ref{REGULNEED0}, we can readily obtain that 
\begin{equation*}
\frac{\widehat{G}_{J_{0},J_{L},s}^{M}\left( \alpha \right) }{%
G_{J_{0},J_{L},s}^{M}\left( \alpha \right) }-1=\frac{\sum_{j=J_{0}}^{J_{L}}%
\frac{\sum_{k}\beta _{jk;p}^{2}}{K_{j}^{M}\left( \alpha \right) }%
U_{j;s}\left( \alpha \right) }{\sum_{j=J_{0}}^{J_{L}}N_{j}\frac{%
G_{0}K_{j}^{M}\left( \alpha _{0}\right) }{K_{j}^{M}\left( \alpha \right) }%
U_{j;s}\left( \alpha \right) }-1
\end{equation*}%
\begin{equation*}
=\frac{\sum_{j=J_{0}}^{J_{L}}\sqrt{N_{j}}\frac{K_{j}^{M}\left( \alpha
_{0}\right) }{K_{j}^{M}\left( \alpha \right) }U_{j;s}\left( \alpha \right)
\left( \frac{1}{\sqrt{N_{j}}}\sum_{k}\left( \frac{\beta _{jk;p}^{2}}{%
G_{0}K_{j}^{M}\left( \alpha _{0}\right) }-1\right) \right) }{%
\sum_{j=J_{0}}^{J_{L}}N_{j}\frac{K_{j}^{M}\left( \alpha _{0}\right) }{%
K_{j}^{M}\left( \alpha \right) }U_{j;s}\left( \alpha \right) }\text{ ,}
\end{equation*}%
so that 
\begin{eqnarray*}
&&\mathbb{P}\left( \left\vert \frac{\sum_{j=J_{0}}^{J_{L}}\sqrt{N_{j}}\frac{%
K_{j}^{M}\left( \alpha _{0}\right) }{K_{j}^{M}\left( \alpha \right) }%
U_{j;s}\left( \alpha \right) \left( \frac{1}{\sqrt{N_{j}}}\sum_{k}\left( 
\frac{\beta _{jk;p}^{2}}{G_{0}K_{j}^{M}\left( \alpha _{0}\right) }-1\right)
\right) }{\sum_{j=J_{0}}^{J_{L}}N_{j}\frac{K_{j}^{M}\left( \alpha
_{0}\right) }{K_{j}^{M}\left( \alpha \right) }U_{j;s}\left( \alpha \right) }%
\right\vert >\delta _{\varepsilon }\right) \\
&\leq &\mathbb{P}\left( \left( J_{L}+J_{0}+1\right) \left\vert \frac{%
\sum_{j=J_{0}}^{J_{L}}\sqrt{N_{j}}\frac{K_{j}^{M}\left( \alpha _{0}\right) }{%
K_{j}^{M}\left( \alpha \right) }U_{j;s}\left( \alpha \right) }{%
\sum_{j=J_{0}}^{J_{L}}N_{j}\frac{K_{j}^{M}\left( \alpha _{0}\right) }{%
K_{j}^{M}\left( \alpha \right) }U_{j;s}\left( \alpha \right) }\right\vert
\right. \\
&&\times \left. \frac{\sup_{j}\left( \frac{1}{\sqrt{N_{j}}}\sum_{k}\left( 
\frac{\beta _{jk;p}^{2}}{G_{0}K_{j}^{M}\left( \alpha _{0}\right) }-1\right)
\right) }{\left( J_{L}+J_{0}+1\right) }>\delta _{\varepsilon }\right) \text{
.}
\end{eqnarray*}%
In view of (\ref{KM1}) and (\ref{KM2}), we obtain%
\begin{equation*}
\frac{\sum_{j=J_{0}}^{J_{L}}\sqrt{N_{j}}\frac{K_{j}^{M}\left( \alpha
_{0}\right) }{K_{j}^{M}\left( \alpha \right) }U_{j;s}\left( \alpha \right) }{%
\sum_{j=J_{0}}^{J_{L}}N_{j}\frac{K_{j}^{M}\left( \alpha _{0}\right) }{%
K_{j}^{M}\left( \alpha \right) }U_{j;s}\left( \alpha \right) }
\end{equation*}%
\begin{eqnarray*}
&=&\frac{\sum_{j=J_{0}}^{J_{L}}B^{j\left( 1+\alpha -\alpha _{0}\right) }j^{s}%
}{\sum_{j=J_{0}}^{J_{L}}B^{j\left( 2+\alpha -\alpha _{0}\right) }j^{s}} \\
&=&\frac{B^{\left( 2+\alpha -\alpha _{0}\right) }-1}{B\left( B^{\left(
1+\alpha -\alpha _{0}\right) }-1\right) }\frac{J_{L}B^{J_{L}\left( 1+\alpha
-\alpha _{0}\right) }-J_{0}B^{J_{0}\left( 1+\alpha -\alpha _{0}\right) -1}}{%
J_{L}B^{J_{L}\left( 2+\alpha -\alpha _{0}\right) }-J_{0}B^{J_{0}\left(
2+\alpha -\alpha _{0}\right) -1}} \\
&=&O\left( B^{-J_{L}}\right) \text{ ,}
\end{eqnarray*}%
so that 
\begin{equation*}
\sup_{j}\left\vert \left( J_{L}+J_{0}+1\right) \frac{\sum_{j=J_{0}}^{J_{L}}%
\sqrt{N_{j}}\frac{K_{j}^{M}\left( \alpha _{0}\right) }{K_{j}^{M}\left(
\alpha \right) }U_{j;s}\left( \alpha \right) }{\sum_{j=J_{0}}^{J_{L}}N_{j}%
\frac{K_{j}^{M}\left( \alpha _{0}\right) }{K_{j}^{M}\left( \alpha \right) }%
U_{j;s}\left( \alpha \right) }\right\vert <+\infty \text{ .}
\end{equation*}%
On the other hand, we have by Chebyshev's inequality and Lemma \ref{Gbehave}
that 
\begin{equation*}
\mathbb{P}\left( \left\vert \left( \frac{1}{\sqrt{N_{j}}}\sum_{k}\left( 
\frac{\beta _{jk;p}^{2}}{G_{0}K_{j}^{M}\left( \alpha _{0}\right) }-1\right)
\right) \right\vert >\delta _{\varepsilon }\left( J_{L}+J_{0}+1\right)
^{2}\right)
\end{equation*}%
\begin{eqnarray*}
&\leq &\frac{1}{\delta _{\varepsilon }^{2}\left( J_{L}+J_{0}+1\right) ^{2}}%
Var\left( \frac{1}{\sqrt{N_{j}}}\sum_{k}\left( \frac{\beta _{jk;p}^{2}}{%
G_{0}K_{j}^{M}\left( \alpha _{0}\right) }-1\right) \right) \\
&=&O\left( \frac{1}{\left( J_{L}+J_{0}+1\right) ^{2}}\right) \text{ ,}
\end{eqnarray*}%
whence 
\begin{equation*}
\mathbb{P}\left( \sup_{j=J_{0},...J_{L}}\left\vert \left( \frac{1}{\sqrt{%
N_{j}}}\sum_{k}\left( \frac{\beta _{jk;p}^{2}}{G_{0}K_{j}^{M}\left( \alpha
_{0}\right) }-1\right) \right) \right\vert >\delta _{\varepsilon }\left(
J_{L}+J_{0}+1\right) ^{2}\right)
\end{equation*}%
\begin{equation*}
\leq \left( J_{L}+J_{0}+1\right) \sup_{j=J_{0},...J_{L}}\mathbb{P}\left(
\left\vert \left( \frac{1}{\sqrt{N_{j}}}\sum_{k}\left( \frac{\beta
_{jk;p}^{2}}{G_{0}K_{j}^{M}\left( \alpha _{0}\right) }-1\right) \right)
\right\vert >\delta _{\varepsilon }\left( J_{L}+J_{0}+1\right) ^{2}\right)
\end{equation*}%
\begin{equation*}
\leq O\left( \frac{1}{\left( J_{L}+J_{0}+1\right) }\right) \text{ .}
\end{equation*}
\end{proof}

Let us focus now our attention on consistency, following a technique
developed in \cite{brillinger} and used in \cite{Robinson} for long memory
processes (see also \cite{dlm} and \cite{dlm2}).

\begin{theorem}
\label{Mconsistency} Assume Condition \ref{REGULNEED0} holds with $%
0<4p+2-\alpha _{0}\leq Q$, we have, as$\ J_{L}\rightarrow \infty $, 
\begin{eqnarray*}
\widehat{\alpha }_{J_{0},J_{L}}^{M} &\rightarrow &_{p}\alpha _{0}\text{ ,} \\
\widehat{G}_{J_{0},J_{L}}^{M}\ &\rightarrow &_{p}G_{0}\text{ .}
\end{eqnarray*}
\end{theorem}

\begin{proof}
Following \cite{Robinson} (see also \cite{dlm2} for the standard needlet
case), we let 
\begin{eqnarray*}
\Delta R_{J_{0},J_{L}}^{M}\left( \alpha ,\alpha _{0}\right)
&=&R_{J_{0},J_{L}}^{M}\left( \alpha \right) -R_{J_{0},J_{L}}^{M}\left(
\alpha _{0}\right) \\
&=&\log \frac{\widehat{G}_{J_{0},J_{L}}^{M}\left( \alpha \right) }{%
G_{J_{0},J_{L}}^{M}\left( \alpha \right) }-\log \frac{\widehat{G}%
_{J_{0},J_{L}}^{M}\left( \alpha _{0}\right) }{G_{0}} \\
&&+\log \frac{G_{J_{0},J_{L}}^{M}\left( \alpha \right) }{G_{0}}+\frac{1}{%
\sum_{j=J_{0}}^{J_{L}}N_{j}}\sum_{j=J_{0}}^{J_{L}}N_{j}\log \frac{%
K_{j}^{M}\left( \alpha \right) }{K_{j}^{M}\left( \alpha _{0}\right) } \\
&=&U_{J_{0},J_{L}}^{M}\left( \alpha \right) -T_{J_{0},J_{L}}^{M}\left(
\alpha \right) \text{ ,}
\end{eqnarray*}%
where 
\begin{eqnarray}
T_{J_{0},J_{L}}^{M}\left( \alpha \right) &=&\log \frac{\widehat{G}%
_{J_{0},J_{L}}^{M}\left( \alpha \right) }{G_{J_{0},J_{L}}^{M}\left( \alpha
\right) }-\log \frac{\widehat{G}_{J_{0},J_{L}}^{M}\left( \alpha _{0}\right) 
}{G_{0}}\text{ ,}  \label{T_def} \\
U_{J_{0},J_{L}}^{M}\left( \alpha \right) &=&\log \frac{G_{J_{0},J_{L}}^{M}%
\left( \alpha \right) }{G_{0}}+\frac{1}{\sum_{j=J_{0}}^{J_{L}}N_{j}}%
\sum_{j=J_{0}}^{J_{L}}N_{j}\log \frac{K_{j}^{M}\left( \alpha \right) }{%
K_{j}^{M}\left( \alpha _{0}\right) }\text{ .}  \label{U_def}
\end{eqnarray}%
For any $\varepsilon >0,$ we have%
\begin{eqnarray*}
\mathbb{P}\left( \left\vert \widehat{\alpha }_{J_{0},J_{L}}^{M}-\alpha
_{0}\right\vert >\varepsilon \right) &=&\mathbb{P}\left( \min_{\left\vert
\alpha -\alpha _{0}\right\vert >\varepsilon }\Delta
R_{J_{0},J_{L}}^{M}\left( \alpha ,\alpha _{0}\right) \leq 0\right) \\
&=&\mathbb{P}\left( \min_{\left\vert \alpha -\alpha _{0}\right\vert
>\varepsilon }T_{J_{0},J_{L}}^{M}\left( \alpha \right)
+U_{J_{0},J_{L}}^{M}\left( \alpha \right) \leq 0\right) \text{ .}
\end{eqnarray*}%
Hence, by combining Lemmas \ref{lemmaU} and \ref{lemmaT}, we obtain 
\begin{equation*}
\lim_{J_{L}\rightarrow +\infty }U_{J_{0},J_{L}}^{M}\left( \alpha ,\alpha
_{0}\right) >0\text{ ,}
\end{equation*}%
\begin{equation*}
\sup_{\alpha }\left\vert T_{J_{0},J_{L}}^{M}\left( \alpha ,\alpha
_{0}\right) \right\vert =o_{p}\left( 1\right) \text{ ,}
\end{equation*}%
as claimed.
\end{proof}

\begin{lemma}
\label{lemmaU}Let $U_{J_{0},J_{L}}^{M}\left( \alpha ,\alpha _{0}\right) $ be
defined as in (\ref{U_def}). For all $\varepsilon <\alpha _{0}-\alpha <2$%
\begin{equation*}
\lim_{J_{L}\rightarrow +\infty }U_{J_{0},J_{L}}^{M}\left( \alpha ,\alpha
_{0}\right)
\end{equation*}%
\begin{equation*}
=\lim_{J_{L}\rightarrow +\infty }\left( \log \frac{1}{%
\sum_{j=J_{0}}^{J_{L}}N_{j}}\sum_{j=J_{0}}^{J_{L}}N_{j}\frac{K_{j}^{M}\left(
\alpha _{0}\right) }{K_{j}^{M}\left( \alpha \right) }-\frac{1}{%
\sum_{j=J_{0}}^{J_{L}}N_{j}}\sum_{j=J_{0}}^{J_{L}}N_{j}\log \frac{%
K_{j}^{M}\left( \alpha _{0}\right) }{K_{j}^{M}\left( \alpha \right) }\right)
\end{equation*}%
\begin{equation*}
=\log \frac{B^{2}-1}{B^{\left( 2+\alpha -\alpha _{0}\right) }-1}+\log B\frac{%
B^{2}}{B^{2}-1}\alpha -\alpha _{0}>\delta _{\varepsilon }>0\text{ .}
\end{equation*}%
if $\alpha _{0}-\alpha =2$ we have 
\begin{equation*}
\lim_{J_{L}\rightarrow +\infty }\frac{1}{\log J_{L}}U_{J_{0},J_{L}}\left(
\alpha ,\alpha _{0}\right) =1
\end{equation*}%
and if $\alpha _{0}-\alpha >2$ we have%
\begin{equation*}
\lim_{J_{L}\rightarrow +\infty }\frac{1}{\log B^{J_{L}}}U_{J_{0},J_{L}}%
\left( \alpha ,\alpha _{0}\right) =\frac{\alpha _{0}-\alpha }{2}-1
\end{equation*}
\end{lemma}

\begin{proof}
Consider first the case $\varepsilon <\alpha _{0}-\alpha <2$. For the sake
of simplicity, we fix $J_{0}=-J_{L}$. We have that%
\begin{equation*}
\frac{1}{\sum_{j=-J_{l}}^{J_{L}}N_{j}}\sum_{j=-J_{L}}^{J_{L}}N_{j}\log \frac{%
K_{j}^{M}\left( \alpha _{0}\right) }{K_{j}^{M}\left( \alpha \right) }
\end{equation*}%
\begin{equation*}
=\frac{1}{\sum_{j=-J_{L}}^{J_{L}}N_{j}}\sum_{j=-J_{L}}^{J_{L}}N_{j}\left(
\log B^{\left( \alpha -\alpha _{0}\right) j}I_{p}\left( B,\alpha -\alpha
_{0}\right) +o\left( j\right) \right)
\end{equation*}%
\begin{equation*}
=\left( \alpha -\alpha _{0}\right) \log B\left( J_{L}-\frac{1}{B^{2}-1}%
\right) +\log \left( I_{p}\left( B,\alpha -\alpha _{0}\right) \right)
+o_{J_{L}}\left( 1\right) \text{ .}
\end{equation*}%
On the other hand, we have%
\begin{equation*}
\log \frac{1}{\sum_{j=-J_{L}}^{J_{L}}B^{2j}}\sum_{j=-J_{L}}^{J_{L}}N_{j}%
\frac{K_{j}^{M}\left( \alpha _{0}\right) }{K_{j}^{M}\left( \alpha \right) }
\end{equation*}%
\begin{equation*}
=\log \frac{I_{p}\left( B,\alpha -\alpha _{0}\right) }{%
\sum_{j=-J_{L}}^{J_{L}}B^{2j}}\sum_{j=-J_{L}}^{J_{L}}B^{2j}B^{\left( \alpha
-\alpha _{0}\right) j}+o_{J_{L}}\left( 1\right)
\end{equation*}%
\begin{equation*}
=\log \frac{B^{2}-1}{B^{2+\left( \alpha -\alpha _{0}\right) }-1}B^{\left(
\alpha -\alpha _{0}\right) \left( J_{L}+1\right) }+\log \left( I_{p}\left(
B,\alpha -\alpha _{0}\right) \right) +o_{J_{L}}\left( 1\right) \text{ }
\end{equation*}%
\begin{equation*}
=\log \frac{B^{2}-1}{B^{2+\left( \alpha -\alpha _{0}\right) }-1}+\left(
\alpha -\alpha _{0}\right) \left( J_{L}+1\right) \log B+\log \left(
I_{p}\left( B,\alpha -\alpha _{0}\right) \right) \text{ .}
\end{equation*}%
As shown in \cite{dlm2}, we have that the function 
\begin{equation*}
l\left( x\right) :=\frac{B^{2}-1}{B^{2+x}-1}+x\left( \frac{B^{2}\log B}{%
B^{2}-1}\right)
\end{equation*}%
has a unique minimum $0$ at $x=0.$\ Therefore, for any $\left\vert \alpha
-\alpha _{0}\right\vert >\varepsilon >0,$ there exists a constant $\delta
_{\varepsilon }>0,$ such that%
\begin{equation*}
U_{J_{0},J_{L}}\left( \alpha ,\alpha _{0}\right) >\delta _{\varepsilon }\ 
\text{.}
\end{equation*}%
If $\alpha -\alpha _{0}<-2$,$\,$we have%
\begin{equation*}
\frac{1}{\log B^{2J_{L}}}U_{J_{0},J_{L}}\left( \alpha ,\alpha _{0}\right)
\end{equation*}%
\begin{eqnarray*}
&=&\frac{1}{\log B^{2J_{L}}}\left\{ \log \left[ \sum_{j=J_{0}}^{J_{L}}B^{j%
\left( 2+\alpha -\alpha _{0}\right) }\right] -\log B^{2J_{L}}-\frac{\left(
\alpha -\alpha _{0}\right) }{\sum_{j=J_{0}}^{J_{L}}N_{j}}%
\sum_{j=J_{0}}^{J_{L}}N_{j}\log \frac{K_{j}^{M}\left( \alpha _{0}\right) }{%
K_{j}^{M}\left( \alpha \right) }\right\} +o_{J_{L}}\left( 1\right) \\
&=&\frac{\alpha _{0}-\alpha }{2}-1\text{ .}
\end{eqnarray*}%
Finally, we have for $\alpha -\alpha _{0}=-2$%
\begin{equation*}
\lim_{J_{L}\rightarrow \infty }\frac{1}{\log J_{L}}U_{J_{0},J_{L}}\left(
\alpha ,\alpha _{0}\right)
\end{equation*}%
\begin{equation*}
\lim_{J_{L}\rightarrow \infty }\frac{1}{\log J_{L}}\left\{ -\log
B^{2J_{L}}+\log J_{L}+O_{J_{L}}\left( 1\right) +\log
B^{2J_{L}}O_{J_{L}}\left( 1\right) \right\} =1\text{ .}
\end{equation*}
\end{proof}

\begin{lemma}
\label{lemmaT}As $J_{L}\rightarrow +\infty $, we have 
\begin{equation*}
\sup_{\alpha }\left\vert T_{J_{0},J_{L}}^{M}\left( \alpha ,\alpha
_{0}\right) \right\vert =o_{p}\left( 1\right) \text{ .}
\end{equation*}
\end{lemma}

\begin{proof}
Because 
\begin{equation*}
\frac{\widehat{G}_{J_{0},J_{L}}^{M}\left( \alpha \right) }{%
G_{J_{0},J_{L}}^{M}\left( \alpha \right) }=\frac{1}{G_{0}}\frac{%
\sum_{j=J_{0}}^{J_{L}}\frac{\Lambda _{j;p}}{K_{j}^{M}\left( \alpha \right) }%
}{\sum_{j=J_{0}}^{J_{L}}N_{j}\frac{K_{j}^{M}\left( \alpha _{0}\right) }{%
K_{j}^{M}\left( \alpha \right) }}\text{ ,}
\end{equation*}%
it follows from Lemma \ref{Gbehave} that 
\begin{equation*}
\mathbb{E}\left( \frac{\widehat{G}_{J_{0},J_{L}}^{M}\left( \alpha \right) }{%
G_{J_{0},J_{L}}^{M}\left( \alpha \right) }-1\right) =0\text{ ,}
\end{equation*}%
while%
\begin{equation*}
Var\left( \frac{\widehat{G}_{J_{0},J_{L}}^{M}\left( \alpha \right) }{%
G_{J_{0},J_{L}}^{M}\left( \alpha \right) }-1\right) =O\left(
B^{-2J_{L}}\right) \text{ .}
\end{equation*}%
Indeed, we have%
\begin{equation*}
Var\left( \frac{\widehat{G}_{J_{0},J_{L}}^{M}\left( \alpha \right) }{%
G_{J_{0},J_{L}}^{M}\left( \alpha \right) }\right) =\left(
G_{J_{0},J_{L}}^{M}\left( \alpha \right) \right) ^{-2}Var\left( \widehat{G}%
_{J_{0},J_{L}}^{M}\left( \alpha \right) \right)
\end{equation*}%
\begin{equation*}
=\frac{\left( B^{\left( 2+\alpha -\alpha _{0}\right) }-1\right) ^{2}}{\left(
B^{2\left( 1+\alpha -\alpha _{0}\right) }-1\right) }\frac{I_{p}\left(
B,\alpha -\alpha _{0}\right) ^{2}}{B^{\alpha _{0}-\alpha }4^{2p-\alpha _{0}+%
\frac{\alpha }{2}}}\frac{\Gamma \left( 4p+1-\alpha _{0}\right) }{\Gamma
^{2}\left( 2p-\frac{\alpha }{2}+1\right) }B^{-2J_{L}}+o_{J_{L}}\left(
B^{-2J_{L}}\right)
\end{equation*}%
By Chebyshev's inequality we have 
\begin{equation*}
\frac{\widehat{G}_{J_{0},J_{L}}^{M}\left( \alpha \right) }{%
G_{J_{0},J_{L}}^{M}\left( \alpha \right) }-1\rightarrow _{p}0\text{ ,}
\end{equation*}%
and from Slutsky's Lemma%
\begin{equation*}
\log \left( \frac{\widehat{G}_{J_{0},J_{L}}^{M}\left( \alpha \right) }{%
G_{J_{0},J_{L}}^{M}\left( \alpha \right) }-1\right) \rightarrow _{p}0\text{ .%
}
\end{equation*}%
On the other hand, by Lemma \ref{unifconvergence}%
\begin{equation*}
\sup \left\vert \frac{\widehat{G}_{J_{0},J_{L}}^{M}\left( \alpha \right) }{%
G_{J_{0},J_{L}}^{M}\left( \alpha \right) }-1\right\vert \rightarrow _{p}0%
\text{ ,}
\end{equation*}%
as we claimed.
\end{proof}

Our purpose now is to study an asymptotic convergence of estimator $\widehat{%
\alpha }_{J_{0},J_{L}}^{M}$.

\begin{theorem}
\label{prop1}Let $0<4p-\alpha _{0}\leq Q$. Assume Condition \ref{REGULNEED0}
holds with . Hence we have%
\begin{equation}
B^{J_{L}}\left( \widehat{\alpha }_{J}^{M}-\alpha _{0}\right) =O_{p}\left(
1\right) \text{ , as}\ J_{L}\rightarrow \infty \text{ .}  \label{asyngauss0}
\end{equation}%
Under Condition \ref{REGULNEED2}, we have 
\begin{equation}
B^{J_{L}}\left( \widehat{\alpha }_{J}^{M}-\alpha _{0}\right) \rightarrow
_{p}-\frac{I_{p,0}\left( \alpha _{0}+1\right) }{I_{p,0}\left( \alpha
_{0}\right) }\frac{\log B}{\left( B+1\right) }\kappa \text{ .}
\label{asyngauss2}
\end{equation}%
Under Condition \ref{REGULNEED3}, we have%
\begin{equation}
B^{J_{L}}\left( \widehat{\alpha }_{J}^{M}-\alpha _{0}\right) \rightarrow
_{d}N\left( 0,\varsigma _{0}^{2}\right)  \label{asyngauss3}
\end{equation}%
where 
\begin{eqnarray*}
\varsigma _{0}^{2} &:&=\varsigma _{0}^{2}\left( p,B,\alpha _{0}\right)
=\sigma _{0}^{2}\left( 1+\widetilde{\tau }\right) \frac{\left(
B^{2}-1\right) ^{3}}{B^{4}\log ^{2}B}\text{ ,} \\
\sigma _{0}^{2} &:&=\sigma _{0}^{2}\left( p,\alpha _{0}\right) =\frac{2}{%
2^{4p-\alpha _{0}}}\frac{\Gamma \left( 4p+1-\alpha _{0}\right) }{\left(
\Gamma \left( 2p+1-\alpha _{0}/2\right) \right) ^{2}}\text{ ,} \\
\widetilde{\tau } &:&=\frac{1}{B^{2}}\left( \left( B^{2}+1\right) \left( 
\widetilde{\tau }_{0}+\widetilde{\tau }_{2}+\widetilde{\tau }_{0}\widetilde{%
\tau }_{2}\right) +2\widetilde{\tau }_{1}-\widetilde{\tau }_{1}^{2}\right) 
\text{ ,}
\end{eqnarray*}%
\newline
with $\widetilde{\tau }_{0,}\widetilde{\tau }_{1}$ $\widetilde{\tau }_{2}$
as defined in Lemma \ref{Lemmasumtau}.
\end{theorem}

\begin{proof}
Again we shall focus on the Taylor expansion%
\begin{equation*}
0=\frac{d}{d\alpha }R_{J_{0},J_{L}}^{M}\left( \alpha \right) |_{\alpha =%
\widehat{\alpha }_{J_{0},J_{L}}^{M}}=S_{J_{0},J_{L}}^{M}\left( \alpha
\right) |_{\alpha =\alpha _{0}}+Q_{J_{0},J_{L}}^{M}\left( \alpha \right)
|_{\alpha =\overline{\alpha }}\left( \widehat{\alpha }_{J_{0},J_{L}}^{M}-%
\alpha \right) \text{,}
\end{equation*}%
\begin{eqnarray*}
S_{J_{0},J_{L}}^{M}\left( \alpha \right) &=&\frac{d}{d\alpha }%
R_{J_{0},J_{L}}^{M}\left( \alpha \right) \text{ ;} \\
Q_{J_{0},J_{L}}^{M}\left( \alpha \right) &=&\frac{d^{2}}{d\alpha ^{2}}%
R_{J_{0},J_{L}}^{M}\left( \alpha \right) \text{ ,}
\end{eqnarray*}%
where $\overline{\alpha }\in \left[ \alpha _{0}-\delta _{J_{L}},\alpha
_{0}+\delta _{J_{L}}\right] ,$ and $\delta _{J}\rightarrow _{p}0$ as $%
J_{L}\rightarrow \infty $ by Lemma \ref{Mconsistency}. The equation above
then leads to%
\begin{equation*}
\left( \widehat{\alpha }_{J_{0},J_{L}}^{M}-\alpha _{0}\right)
=-S_{J_{0},J_{L}}^{M}\left( \alpha _{0}\right) \left(
Q_{J_{0},J_{L}}^{M}\left( \overline{\alpha }\right) \right) ^{-1}\text{ .}
\end{equation*}

The proof is readily completed by combining Lemma \ref{Mconsistency1} and %
\ref{Mconsistency2}.
\end{proof}

\begin{lemma}
\label{Mconsistency1}Assume Condition \ref{REGULNEED2} holds with $%
0<4p+2-\alpha _{0}\leq M$, we have 
\begin{equation*}
B^{J_{L}}S_{J_{0},J_{L}}^{M}\left( \alpha _{0}\right) \rightarrow
_{p}-\kappa \frac{I_{p,0}\left( \alpha _{0}+1\right) }{I_{p,0}\left( \alpha
_{0}\right) }\frac{\log B}{\left( B+1\right) }\text{ ;}
\end{equation*}%
if Condition \ref{REGULNEED3} holds we have 
\begin{equation*}
B^{J_{L}}S_{J_{0},J_{L}}^{M}\left( \alpha _{0}\right) \rightarrow
_{D}N\left( 0,\sigma _{0}^{2}\left( 1+\widetilde{\tau }\right) \frac{\log
^{2}B}{B^{2}-1}\right) \text{ .}
\end{equation*}%
.
\end{lemma}

\begin{proof}
Note, first of all, that, as in \cite{dlm2}, the proof of (\ref{asyngauss2})
is totally equivalent to the case of (\ref{asyngauss0}). First of all, we
can rewrite $S_{J_{0},J_{L}}^{M}\left( \alpha \right) $ as follows.%
\begin{eqnarray*}
S_{J_{0},J_{L}}^{M}\left( \alpha \right) &=&\frac{d}{d\alpha }\log \widehat{G%
}_{J_{0},J_{L}}^{M}\left( \alpha \right) -\frac{d}{d\alpha }\frac{1}{%
\sum_{j=J_{0}}^{J_{L}}N_{j}}\sum_{j=J_{0}}^{J_{L}}\sum_{k}\log \frac{\beta
_{jk;p}^{2}}{K_{j}^{M}\left( \alpha \right) } \\
&=&\frac{\widehat{G}_{J_{0},J_{L},1}^{M}\left( \alpha \right) }{\widehat{G}%
_{J_{0},J_{L}}^{M}\left( \alpha \right) }-\frac{1}{%
\sum_{j=J_{0}}^{J_{L}}N_{j}}\sum_{j=J_{0}}^{J_{L}}\sum_{k}\frac{%
K_{j,1}^{M}\left( \alpha \right) }{K_{j}^{M}\left( \alpha \right) } \\
&=&\frac{1}{\sum_{j=J_{0}}^{J_{L}}N_{j}}\sum_{j=J_{0}}^{J_{L}}\left( \frac{%
K_{j,1}^{M}\left( \alpha \right) }{K_{j}^{M}\left( \alpha \right) }\right)
\sum_{k}\left( \frac{\beta _{jk;p}^{2}}{\widehat{G}_{J_{0},J_{L}}^{M}\left(
\alpha \right) K_{j}^{M}\left( \alpha \right) }-1\right) \text{ .}
\end{eqnarray*}%
We can easily see that 
\begin{equation*}
S_{J_{0},J_{L}}^{M}\left( \alpha _{0}\right) =\frac{G_{0}}{\widehat{G}%
_{J_{0},J_{L}}^{M}\left( \alpha _{0}\right) }\overline{S}_{J_{0},J_{L}}^{M}%
\left( \alpha _{0}\right) \text{ ,}
\end{equation*}%
where%
\begin{equation*}
\overline{S}_{J_{0},J_{L}}^{M}\left( \alpha _{0}\right) =\frac{1}{%
\sum_{j=J_{0}}^{J_{L}}N_{j}}\sum_{j=J_{0}}^{J_{L}}\frac{K_{j,1}^{M}\left(
\alpha _{0}\right) }{K_{j}^{M}\left( \alpha _{0}\right) }\sum_{k}\left( 
\frac{\beta _{jk;p}^{2}}{G_{0}K_{j}^{M}\left( \alpha _{0}\right) }-\frac{%
\widehat{G}_{J_{0},J_{L}}^{M}\left( \alpha _{0}\right) }{G_{0}}\right) \text{
}
\end{equation*}%
and from Lemma \ref{unifconvergence} we have 
\begin{equation*}
\frac{G_{0}}{\widehat{G}_{J_{0},J_{L},1}^{M}\left( \alpha _{0}\right) }%
\rightarrow _{p}1\text{ .}
\end{equation*}%
Under Condition \ref{REGULNEED2} we have%
\begin{equation*}
\lim_{J_{L}\rightarrow \infty }B^{J_{L}}\mathbb{E}\left( \overline{S}%
_{J_{0},J_{L}}^{M}\left( \alpha _{0}\right) \right)
\end{equation*}%
\begin{eqnarray*}
&=&\lim_{J_{L}\rightarrow \infty }\frac{B^{J_{L}}}{%
\sum_{j=J_{0}}^{J_{L}}N_{j}}\sum_{j=J_{0}}^{J_{L}}\left( -\frac{%
K_{j,1}^{M}\left( \alpha _{0}\right) }{K_{j}^{M}\left( \alpha _{0}\right) }%
\right) \left( \frac{\mathbb{E}\left( \widehat{\Lambda }_{j;p}\right) }{%
G_{0}K_{j}^{M}\left( \alpha _{0}\right) }-\frac{N_{j}}{%
\sum_{j=J_{0}}^{J_{L}}N_{j}}\sum_{j=J_{0}}^{J_{L}}\frac{\mathbb{E}\left( 
\widehat{\Lambda }_{j;p}\right) }{G_{0}K_{j}^{M}\left( \alpha _{0}\right) }%
\right) \\
&=&\lim_{J_{L}\rightarrow \infty }\frac{I_{p,0}\left( \alpha _{0}+1\right) }{%
I_{p,0}\left( \alpha _{0}\right) }\frac{\kappa B^{J_{L}}}{%
\sum_{j=J_{0}}^{J_{L}}B^{2j}}\sum_{j=J_{0}}^{J_{L}}\log B^{j}\cdot
B^{2j}\left( B^{-j}-\frac{1}{\sum_{j=J_{0}}^{J_{L}}B^{2j}}%
\sum_{j=J_{0}}^{J_{L}}B^{j}\right) +o_{J_{L}}\left( 1\right) \\
&=&\lim_{J_{L}\rightarrow \infty }-\kappa \frac{I_{p,0}\left( \alpha
_{0}+1\right) }{I_{p,0}\left( \alpha _{0}\right) }\frac{\log B}{\left(
B+1\right) }+o_{J_{L}}\left( 1\right) \text{ ;}
\end{eqnarray*}%
while under Condition \ref{REGULNEED3} we find immediately 
\begin{equation*}
\mathbb{E}\left( \overline{S}_{J_{0},J_{L}}^{M}\left( \alpha _{0}\right)
\right) =o_{J_{L}}\left( 1\right) \text{ .}
\end{equation*}%
In order to compute the variance of $\overline{S}_{J_{0},J_{L}}^{M}\left(
\alpha _{0}\right) $, we split it into 3 terms (see again \cite{dlm2}):%
\begin{equation*}
Var\left( \overline{S}_{J_{0},J_{L}}^{M}\left( \alpha _{0}\right) \right)
=A+B+C\text{ ,}
\end{equation*}%
where%
\begin{equation*}
A=\frac{1}{\left( \sum_{j=J_{0}}^{J_{L}}N_{j}\right) ^{2}}%
\sum_{j_{1}}\sum_{j_{2}}\left( \frac{K_{j_{1},1}^{M}\left( \alpha
_{0}\right) }{K_{j_{1}}^{M}\left( \alpha _{0}\right) }\frac{%
K_{j_{2},1}^{M}\left( \alpha _{0}\right) }{K_{j_{2}}^{M}\left( \alpha
_{0}\right) }\right) Cov\left( \frac{\sum_{k_{1}}\beta _{j_{1}k_{1};p}^{2}}{%
G_{0}K_{j_{1}}^{M}\left( \alpha _{0}\right) },\frac{\sum_{k_{2}}\beta
_{j_{2}k_{2};p}^{2}}{G_{0}K_{j_{2}}^{M}\left( \alpha _{0}\right) }\right) 
\text{ ,}
\end{equation*}%
\begin{equation*}
B=\frac{1}{\left( \sum_{j=J_{0}}^{J_{L}}N_{j}\right) ^{2}}%
\sum_{j_{1}}\sum_{j_{2}}\left( \frac{K_{j_{1},1}^{M}\left( \alpha
_{0}\right) }{K_{j_{1}}^{M}\left( \alpha _{0}\right) }\frac{%
K_{j_{2},1}^{M}\left( \alpha _{0}\right) }{K_{j_{2}}^{M}\left( \alpha
_{0}\right) }\right) N_{j_{1}}N_{j_{2}}Var\left( \frac{\widehat{G}%
_{J_{0},J_{L}}^{M}\left( \alpha _{0}\right) }{G_{0}}\right) \text{ ,}
\end{equation*}%
\begin{equation*}
C=\frac{-2}{\left( \sum_{j=J_{0}}^{J_{L}}N_{j}\right) ^{2}}%
\sum_{j_{1}}\sum_{j_{2}}\left( \frac{K_{j_{1},1}^{M}\left( \alpha
_{0}\right) }{K_{j_{1}}^{M}\left( \alpha _{0}\right) }\frac{%
K_{j_{2},1}^{M}\left( \alpha _{0}\right) }{K_{j_{2}}^{M}\left( \alpha
_{0}\right) }\right) Cov\left( \frac{\sum_{k_{1}}\beta _{j_{1}k_{1};p}^{2}}{%
G_{0}K_{j}^{M}\left( \alpha _{0}\right) },N_{j_{2}}\frac{\widehat{G}%
_{J_{0},J_{L}}^{M}\left( \alpha _{0}\right) }{G_{0}}\right) \text{ .}
\end{equation*}%
By fixing $j=j_{1}$, $\Delta j=j_{2}-j_{1}$, we have: 
\begin{eqnarray*}
A &=&\frac{1}{\left( \sum_{j=J_{0}}^{J_{L}}N_{j}\right) ^{2}}%
\sum_{j=J_{0}}^{J_{L}}\sum_{\Delta j=J_{0}-j}^{J_{L}-j}\left( \frac{%
K_{j,1}^{M}\left( \alpha _{0}\right) }{K_{j}^{M}\left( \alpha _{0}\right) }%
\frac{K_{j+\Delta j,1}^{M}\left( \alpha _{0}\right) }{K_{j+\Delta
j}^{M}\left( \alpha _{0}\right) }\right) \\
&&\times Cov\left( \frac{\sum_{k_{1}}\beta _{jk_{1};p}^{2}}{%
G_{0}K_{j}^{M}\left( \alpha _{0}\right) },\frac{\sum_{k_{2}}\beta _{j+\Delta
jk_{2};p}^{2}}{G_{0}K_{j+\Delta j}^{M}\left( \alpha _{0}\right) }\right) \\
&=&\frac{\left( B^{2}-1\right) \log ^{2}B}{B^{2}B^{2J_{L}}}\sigma
_{0}^{2}\left( \left( 1+\tau _{0}\right) J_{L}^{2}-\left( \frac{2}{B^{2}-1}%
\left( 1+\widetilde{\tau }_{1}\right) \right) J_{L}\right. \\
&&+\left. \frac{B^{2}+1}{\left( B^{2}-1\right) ^{2}}\left( 1+\widetilde{\tau 
}_{2}\right) \right) +o\left( B^{-2J_{L}}\right) \text{ ,}
\end{eqnarray*}%
by applying Lemmas \ref{Lemmasumtau} and \ref{lemmasums}. On the other hand,
from Lemma \ref{Gbehave}, we obtain%
\begin{eqnarray*}
B &=&\frac{1}{\left( \sum_{j=J_{0}}^{J_{L}}N_{j}\right) ^{2}}%
\sum_{j_{1}}\sum_{j_{2}}\left( \frac{K_{j_{1},1}^{M}\left( \alpha
_{0}\right) }{K_{j_{1}}^{M}\left( \alpha _{0}\right) }\frac{%
K_{j_{2},1}^{M}\left( \alpha _{0}\right) }{K_{j_{2}}^{M}\left( \alpha
_{0}\right) }\right) N_{j_{1}}N_{j_{2}}Var\left( \frac{\widehat{G}%
_{J_{0},J_{L}}^{M}\left( \alpha _{0}\right) }{G_{0}}\right) \\
&=&\frac{\sigma _{0}^{2}}{\left( \sum_{j=J_{0}}^{J_{L}}B^{2j}\right) ^{4}}%
\left( \sum_{j_{1}=J_{0}}^{J_{L}}\log B^{j_{1}}B^{2j_{1}}\right) \left(
\sum_{j_{2}=J_{0}}^{J_{L}}\log B^{j_{2}}B^{2j_{2}}\right)
\sum_{j_{1}=J_{0}}^{J_{L}}B^{2j}\sum_{\Delta j=-J_{0}-j}^{J_{L}-j}B^{\alpha
_{0}\Delta j}\tau _{B}\left( \Delta j\right) \\
&=&\frac{\sigma _{0}^{2}}{\left( \sum_{j=J_{0}}^{J_{L}}B^{2j}\right) ^{4}}%
\left( \sum_{j=J_{0}}^{J_{L}}\log B^{j}B^{2j}\right) ^{2}\left( \frac{B^{2}}{%
B^{2}-1}\left( 1+\widetilde{\tau }_{0}\right) B^{2J_{L}}\right) +o\left(
B^{2J_{L}}\right) \\
&=&\frac{\left( B^{2}-1\right) \log ^{2}B}{B^{2}}\sigma
_{0}^{2}B^{-2J_{L}}\left( \left( J_{L}-\frac{1}{B^{2}-1}\right) ^{2}\left( 1+%
\widetilde{\tau }_{0}\right) \right) \text{ .}
\end{eqnarray*}%
Finally, we have that 
\begin{eqnarray*}
C &=&\frac{-2}{\left( \sum_{j=J_{0}}^{J_{L}}N_{j}\right) ^{3}}\left(
\sum_{j=J_{0}}^{J_{L}}\log B^{j}Cov\left( \frac{\sum_{k_{1}}\beta
_{j_{1}k_{1};p}^{2}}{G_{0}K_{j}^{M}\left( \alpha _{0}\right) }%
,\sum_{j3=J_{0}}^{J_{L}}\frac{\sum_{k_{1}}\beta _{j_{3}k_{1};p}^{2}}{%
G_{0}K_{j_{3}}^{M}\left( \alpha _{0}\right) }\right) \right) \\
&&\times \left( \sum_{j_{2}=J_{0}}^{J_{L}}B^{2j_{2}}\log B^{j_{2}}\right)
+o\left( B^{2J_{L}}\right) \\
&=&-2\sigma _{0}^{2}\frac{\left( B^{2}-1\right) \log ^{2}B}{B^{2}}%
B^{-2J_{L}}\left( \left( \left( 1+\widetilde{\tau }_{0}\right) J_{L}-\left( 
\frac{1}{B^{2}-1}\left( 1+\widetilde{\tau }_{1}\right) \right) \right)
\right. \\
&&\times \left. \left( J_{L}-\frac{1}{B^{2}-1}\right) \right) +o\left(
B^{2J_{L}}\right) \text{.}
\end{eqnarray*}%
Summing all these terms, we obtain 
\begin{equation*}
A+B+C=\sigma _{0}^{2}\frac{B^{2}\log ^{2}B}{\left( B^{2}-1\right) }\left( 1+%
\widetilde{\tau }\right) B^{-2J_{L}}+o\left( B^{2J_{L}}\right) \text{ .}
\end{equation*}%
We can use the Lemma \ref{Lemmasumtau} to observe that%
\begin{equation*}
Var\left( \overline{S}_{J_{0},J_{L}}^{M}\left( \alpha _{0}\right) \right) =%
\frac{\sigma _{0}^{2}\left( 1+\widetilde{\tau }\right) }{\left(
\sum_{j=J_{0}}^{J_{L}}B^{2j}\right) ^{3}}\left( Z_{J_{L}}+o_{J_{L}}\left(
1\right) \right) \text{.}
\end{equation*}%
Hence we have 
\begin{equation*}
\lim_{J_{L}\rightarrow \infty }B^{2J_{L}}Var\left( \overline{S}%
_{J_{0},J_{L}}^{M}\left( \alpha _{0}\right) \right) =\frac{\sigma
_{0}^{2}\left( 1+\widetilde{\tau }\right) B^{2}\log ^{2}B}{\left(
B^{2}-1\right) }\text{ .}
\end{equation*}%
To prove (\ref{asyngauss3}), it remains to study the behaviour the fourth
order cumulants, observing that this statistics belong to the second order
Wiener chaos with respect to a Gaussian white noise random measure (see \cite%
{nourdinpeccati}, \cite{dlm2}). Let%
\begin{equation*}
B^{J_{L}}S_{J_{L}}\left( \alpha _{0}\right) =\frac{1}{B^{J_{L}}}%
\sum_{j}\left( A_{j}+B_{j}\right) \text{ ,}
\end{equation*}%
where%
\begin{eqnarray}
A_{j} &=&B^{2j}\log B^{j}\left\{ \frac{\sum_{k}\beta _{jk}^{2}}{%
N_{j}G_{0}K_{j}\left( \alpha _{0}\right) }-1\right\} \text{ ,}
\label{Aj_cum} \\
B_{j} &=&B^{2j}\log B^{j}\left\{ \frac{\widehat{G}_{J_{L}}(\alpha _{0})}{%
G_{0}}-1\right\} \text{ .}  \label{Bj_cum}
\end{eqnarray}%
In the Appendix, Lemma \ref{cumulants} proves that:%
\begin{equation*}
\frac{1}{B^{4J_{L}}}cum\left\{
\sum_{l_{1}}(A_{j_{1}}+B_{j_{1}}),\sum_{l_{2}}(A_{j_{2}}+B_{j_{2}}),%
\sum_{l_{3}}(A_{j_{3}}+B_{j_{3}}),\sum_{l_{4}}(A_{j_{4}}+B_{j_{4}})\right\}
\end{equation*}%
\begin{equation*}
=O_{J_{L}}\left( \frac{J_{L}^{4}\log ^{4}B}{B^{2J_{L}}}\right) \text{ .}
\end{equation*}

Exactly as in \cite{dlm} and \cite{dlm2}, the Central Limit Theorem follows
from results in \cite{nourdinpeccati}.
\end{proof}

\begin{lemma}
\label{Mconsistency2} Assume Condition \ref{REGULNEED0} holds with $%
0<4p+2-\alpha _{0}\leq Q$. Then, for $\overline{\alpha }\in \left[ \alpha
_{0}-\delta _{J_{L}},\alpha _{0}+\delta _{J_{L}}\right] $, we have 
\begin{equation*}
Q_{J_{0},J_{L}}^{M}\left( \overline{\alpha }\right) \rightarrow _{p}\frac{%
B^{2}\log ^{2}B}{\left( B^{2}-1\right) ^{2}}\text{ .}
\end{equation*}
\end{lemma}

\begin{proof}
The procedure is totally analogue to Lemma 19 in \cite{dlm2}. We obtain:

\begin{eqnarray*}
Q_{J_{0},J_{L}}^{M}\left( \alpha \right) &=&\frac{G_{J_{0},J_{L},2}^{M}%
\left( \alpha \right) G_{J_{0},J_{L}}^{M}\left( \alpha \right) -\left(
G_{J_{0},J_{L}1}^{M}\left( \alpha \right) \right) ^{2}}{\left(
G_{J_{0},J_{L}}^{M}\left( \alpha \right) \right) ^{2}} \\
&&+\frac{1}{\sum_{j}N_{j}}\sum_{j}N_{j}\frac{K_{j,2}^{M}\left( \alpha
\right) K_{j}^{M}\left( \alpha \right) -\left( K_{j,1}^{M}\left( \alpha
\right) \right) ^{2}}{\left( K_{j}^{M}\left( \alpha \right) \right) ^{2}}
\end{eqnarray*}%
\begin{equation*}
=\frac{\left( \sum_{j}N_{j}\frac{K_{j}^{M}\left( \alpha _{0}\right) }{%
K_{j}^{M}\left( \alpha \right) }\left( 2\left( \frac{K_{j,1}^{M}\left(
\alpha \right) }{K_{j}^{M}\left( \alpha \right) }\right) ^{2}-\frac{%
K_{j,2}^{M}\left( \alpha \right) }{K_{j}^{M}\left( \alpha \right) }\right)
\right) \left( \sum_{j}N_{j}\frac{K_{j}^{M}\left( \alpha _{0}\right) }{%
K_{j}^{M}\left( \alpha \right) }\right) }{\left( \sum_{j}N_{j}\frac{%
K_{j}^{M}\left( \alpha _{0}\right) }{K_{j}^{M}\left( \alpha \right) }\right)
^{2}}
\end{equation*}%
\begin{equation*}
-\frac{\left( \sum_{j}N_{j}\frac{K_{j}^{M}\left( \alpha _{0}\right) }{%
K_{j}^{M}\left( \alpha \right) }\left( -\frac{K_{j,1}^{M}\left( \alpha
\right) }{K_{j}^{M}\left( \alpha \right) }\right) \right) ^{2}}{\left(
\sum_{j}N_{j}\frac{K_{j}^{M}\left( \alpha _{0}\right) }{K_{j}^{M}\left(
\alpha \right) }\right) ^{2}}+\frac{1}{\sum_{j}N_{j}}\sum_{j}N_{j}\frac{%
K_{j,2}^{M}\left( \alpha \right) K_{j}^{M}\left( \alpha \right) -\left(
K_{j,1}^{M}\left( \alpha \right) \right) ^{2}}{\left( K_{j}^{M}\left( \alpha
\right) \right) ^{2}}\text{ .}
\end{equation*}%
$Q_{J_{0},J_{L}}^{M}\left( \alpha \right) $ can be rewritten as the sum of
three terms:%
\begin{equation*}
Q_{J_{0},J_{L}}^{M}\left( \alpha \right) =Q_{1}\left( \alpha \right)
+Q_{2}\left( \alpha \right) +Q_{3}\left( \alpha \right) \text{ ,}
\end{equation*}%
where:%
\begin{equation*}
Q_{1}\left( \alpha \right) =\frac{Q_{1}^{num}\left( \alpha \right) }{%
Q_{1}^{den}\left( \alpha \right) }
\end{equation*}%
\begin{equation*}
=\frac{\left( \sum_{j}N_{j}\frac{K_{j}^{M}\left( \alpha _{0}\right) }{%
K_{j}^{M}\left( \alpha \right) }\left( \frac{K_{j,1}^{M}\left( \alpha
\right) }{K_{j}^{M}\left( \alpha \right) }\right) ^{2}\right) \left(
\sum_{j}N_{j}\frac{K_{j}^{M}\left( \alpha _{0}\right) }{K_{j}^{M}\left(
\alpha \right) }\right) -\left( \sum_{j}N_{j}\frac{K_{j}^{M}\left( \alpha
_{0}\right) }{K_{j}^{M}\left( \alpha \right) }\left( -\frac{%
K_{j,1}^{M}\left( \alpha \right) }{K_{j}^{M}\left( \alpha \right) }\right)
\right) ^{2}}{\left( \sum_{j}N_{j}\frac{K_{j}^{M}\left( \alpha _{0}\right) }{%
K_{j}^{M}\left( \alpha \right) }\right) ^{2}}\text{ ,}
\end{equation*}%
\begin{equation*}
Q_{2}\left( \alpha \right) =\frac{Q_{2}^{num}\left( \alpha \right) }{%
Q_{2}^{den}\left( \alpha \right) }
\end{equation*}%
\begin{equation*}
=\frac{\left( \sum_{j}N_{j}\right) \left( \sum_{j}N_{j}\frac{K_{j}^{M}\left(
\alpha _{0}\right) }{K_{j}^{M}\left( \alpha \right) }\left( \frac{%
K_{j,1}^{M}\left( \alpha \right) }{K_{j}^{M}\left( \alpha \right) }\right)
^{2}\right) -\left( \sum_{j}N_{j}\frac{K_{j}^{M}\left( \alpha _{0}\right) }{%
K_{j}^{M}\left( \alpha \right) }\right) \left( \sum_{j}N_{j}\left( \frac{%
K_{j,1}^{M}\left( \alpha \right) }{K_{j}^{M}\left( \alpha \right) }\right)
^{2}\right) }{\left( \sum_{j}N_{j}\frac{K_{j}^{M}\left( \alpha _{0}\right) }{%
K_{j}^{M}\left( \alpha \right) }\right) \sum_{j}N_{j}}\text{ ,}
\end{equation*}%
\begin{equation*}
Q_{3}\left( \alpha \right) =\frac{Q_{3}^{num}\left( \alpha \right) }{%
Q_{2}^{den}\left( \alpha \right) }
\end{equation*}%
\begin{equation*}
=\frac{\left( \sum_{j}N_{j}\frac{K_{j}^{M}\left( \alpha _{0}\right) }{%
K_{j}^{M}\left( \alpha \right) }\right) \left( \sum_{j}N_{j}\frac{%
K_{j,2}^{M}\left( \alpha \right) }{K_{j}^{M}\left( \alpha \right) }\right)
-\left( \sum_{j}N_{j}\right) \left( \sum_{j}N_{j}\frac{K_{j}^{M}\left(
\alpha _{0}\right) }{K_{j}^{M}\left( \alpha \right) }\frac{K_{j,2}^{M}\left(
\alpha \right) }{K_{j}^{M}\left( \alpha \right) }\right) }{\left(
\sum_{j}N_{j}\frac{K_{j}^{M}\left( \alpha _{0}\right) }{K_{j}^{M}\left(
\alpha \right) }\right) \sum_{j}N_{j}}\text{ .}
\end{equation*}

The next step consists in showing that: 
\begin{equation*}
Q_{2}\left( \alpha \right) +Q_{3}\left( \alpha \right) =o_{J_{L}}(1)\text{ .}
\end{equation*}

Using Corollary \ref{corkjratio}, $Q_{2}^{num}\left( \alpha \right) $ can be
written as:%
\begin{equation*}
Q_{2}^{num}\left( \alpha \right)
\end{equation*}%
\begin{eqnarray*}
&=&\left( \sum_{j}N_{j}\right) \left( \sum_{j}N_{j}\frac{K_{j}^{M}\left(
\alpha _{0}\right) }{K_{j}^{M}\left( \alpha \right) }\left( \log ^{2}B^{j}+2%
\frac{I_{p,1}\left( B\right) }{I_{p,0}\left( B\right) }\log B^{j}+\left( 
\frac{I_{p,1}\left( B\right) }{I_{p,0}\left( B\right) }\right)
^{2}+o_{J_{L}}(1)\right) \right) \\
&&-\left( \sum_{j}N_{j}\frac{K_{j}^{M}\left( \alpha _{0}\right) }{%
K_{j}^{M}\left( \alpha \right) }\right) \left( \sum_{j}N_{j}\left( \log
^{2}B^{j}+2\frac{I_{p,1}\left( B\right) }{I_{p,0}\left( B\right) }\log
B^{j}+\left( \frac{I_{p,1}\left( B\right) }{I_{p,0}\left( B\right) }\right)
^{2}+o_{J_{L}}(1)\right) \right) \text{ ,}
\end{eqnarray*}

while $Q_{3}^{num}\left( \alpha \right) \,$becomes:%
\begin{equation*}
Q_{3}^{num}\left( \alpha \right)
\end{equation*}%
\begin{eqnarray*}
&=&\left( \sum_{j}N_{j}\frac{K_{j}^{M}\left( \alpha _{0}\right) }{%
K_{j}^{M}\left( \alpha \right) }\right) \left( \sum_{j}N_{j}\left( \log
B^{2j}+2\frac{I_{p,1}\left( B\right) }{I_{p,0}\left( B\right) }\log B^{j}+%
\frac{I_{p,2}\left( B\right) }{I_{p,0}\left( B\right) }+o_{J_{L}}(1)\right)
\right) \\
&&-\left( \sum_{j}N_{j}\right) \left( \sum_{j}N_{j}\frac{K_{j}\left( \alpha
_{0}\right) }{K_{j}\left( \alpha \right) }\left( \log B^{2j}+2\frac{%
I_{p,1}\left( B\right) }{I_{p,0}\left( B\right) }\log B^{j}+\frac{%
I_{p,2}\left( B\right) }{I_{p,0}\left( B\right) }+o_{J_{L}}(1)\right)
\right) \text{ ,}
\end{eqnarray*}%
so that:%
\begin{equation*}
\frac{Q_{2}^{num}\left( \alpha \right) +Q_{3}^{num}\left( \alpha \right) }{%
Q_{2}^{den}\left( \alpha \right) }=o_{J_{L}}(1)\text{ .}
\end{equation*}

It remains to study $Q_{2}^{den}\left( \alpha \right) ;$ by Propositions \ref%
{propsumKj} and \ref{prop27}, we have:%
\begin{equation*}
\lim_{J_{L}\rightarrow \infty }\frac{1}{B^{2\left( 2+\frac{\alpha -a_{0}}{2}%
\right) J_{L}}}\left( \sum_{j}N_{j}\frac{K_{j}^{M}\left( \alpha _{0}\right) 
}{K_{j}^{M}\left( \alpha \right) }\right) \left( \sum_{j}N_{j}\right)
\end{equation*}%
\begin{eqnarray*}
&=&\lim_{J_{L}\rightarrow \infty }\frac{c_{B}^{2}I_{p}\left( B,\alpha
-\alpha _{0}\right) }{B^{2\left( 2+\frac{\alpha -a_{0}}{2}\right) J_{L}}}%
\left( \sum_{j}B^{2j\left( 1+\frac{\alpha -\alpha _{0}}{2}\right) }\right)
\left( \sum_{j}B^{2j}\right) \\
&=&c_{B}^{2}I_{p}\left( B,\alpha -\alpha _{0}\right) \frac{B^{2\left( 1+%
\frac{\alpha -a_{0}}{2}\right) }}{B^{2\left( 1+\frac{\alpha -a_{0}}{2}%
\right) }-1}\frac{B^{2}}{B^{2}-1}>0\text{ .}
\end{eqnarray*}

Finally, we prove that $Q_{1}\left( \overline{\alpha }_{L}\right)
\rightarrow _{p}\frac{B^{2}\log ^{2}B}{\left( B^{2}-1\right) ^{2}}$. Using
again Proposition \ref{propsumKj} and Corollary \ref{corkjratio}, we write
the numerator $Q_{1}^{num}\left( \alpha \right) $ as:%
\begin{equation*}
Q_{1}^{num}\left( \alpha \right)
\end{equation*}%
\begin{eqnarray*}
&=&\left( \sum_{j}\frac{K_{j}^{M}\left( \alpha _{0}\right) }{K_{j}^{M}\left(
\alpha \right) }N_{j}\right) \left( \sum_{j}N_{j}\frac{K_{j}^{M}\left(
\alpha _{0}\right) }{K_{j}^{M}\left( \alpha \right) }\left( \log B^{j}+\frac{%
I_{p,1}\left( B\right) }{I_{p,0}\left( B\right) }\right) ^{2}\right) \\
&&-\left( \sum_{j}N_{j}\frac{K_{j}^{M}\left( \alpha _{0}\right) }{%
K_{j}^{M}\left( \alpha \right) }\left( \log B^{j}+\frac{I_{p,1}\left(
B\right) }{I_{p,0}\left( B\right) }\right) \right) ^{2}
\end{eqnarray*}%
\begin{equation*}
=\left( \sum_{j}\frac{K_{j}^{M}\left( \alpha _{0}\right) }{K_{j}^{M}\left(
\alpha \right) }N_{j}\left( \sum_{j}N_{j}\frac{K_{j}^{M}\left( \alpha
_{0}\right) }{K_{j}^{M}\left( \alpha \right) }\log ^{2}B^{j}\right) \right)
-\left( \sum_{j}N_{j}\frac{K_{j}^{M}\left( \alpha _{0}\right) }{%
K_{j}^{M}\left( \alpha \right) }\log B^{j}\right) ^{2}
\end{equation*}

Let $s=2\left( 1+\frac{\alpha -a_{0}}{2}\right) $; by applying Corollary \ref%
{cor28} we have:%
\begin{eqnarray*}
\lim_{J_{L}\rightarrow \infty }\frac{1}{B^{2sJ_{L}}}Q_{1}^{num}\left( \alpha
\right) &=&\lim_{J_{L}\rightarrow \infty }\frac{c_{B}^{2}I_{p}\left(
B,\alpha -\alpha _{0}\right) }{B^{2sJ_{L}}}Z_{J_{L}}\left( s\right) \\
&=&\log ^{2}B\frac{B^{3s}}{(B^{s}-1)^{4}}c_{B}^{2}I\left( B,\alpha
_{0},\alpha \right) \text{ . }
\end{eqnarray*}

It remains to study $Q_{1}^{den}\left( \alpha \right) ;$ by using again (\ref%
{MKalpha}) and Proposition \ref{prop27}:%
\begin{eqnarray*}
\lim_{J_{L}\rightarrow \infty }\frac{1}{B^{2sJ_{L}}}Q_{1}^{den}\left( \alpha
\right) &=&\lim_{J_{L}\rightarrow \infty }\frac{c_{B}^{2}I_{p}\left(
B,\alpha -\alpha _{0}\right) }{B^{2sJ_{L}}}\left( \sum_{j}B^{sj}\right) ^{2}
\\
&=&c_{B}^{2}I_{p}\left( B,\alpha -\alpha _{0}\right) \left( \frac{B^{s}}{%
B^{s}-1}\right) ^{2}\text{ .}
\end{eqnarray*}%
Hence%
\begin{equation*}
\lim_{J_{L}\rightarrow \infty }Q_{J\,_{0},J_{L}}^{M}\left( \alpha \right) =%
\frac{B^{2\left( 1+\frac{\alpha -a_{0}}{2}\right) }\log ^{2}B}{\left(
B^{2\left( 1+\frac{\alpha -a_{0}}{2}\right) }-1\right) ^{2}}\text{ .}
\end{equation*}%
For the consistency of $\widehat{\alpha }_{L}$, for $\overline{\alpha }%
_{L}\in \left[ \alpha _{0}-\widehat{\alpha }_{L},\alpha _{0}+\widehat{\alpha 
}_{L}\right] $, we have 
\begin{equation*}
Q_{J\,_{0},J_{L}}^{M}\left( \overline{\alpha }_{L}\right) \longrightarrow
_{p}\frac{B^{2}\log ^{2}B}{\left( B^{2}-1\right) ^{2}}\text{ .}
\end{equation*}
\end{proof}

\section{Narrow band estimates \label{narrow}}

From Theorem \ref{prop1}, it is evident that, under Condition \ref%
{REGULNEED2}, the presence of the bias term does not allow asymptotic
inference. As in \cite{dlm} and \cite{dlm2}, we suggest a narrow-band
strategy, developed only on the higher tail of the power spectrum, which
allows us to avoid the problem due to the nuisance parameter. We start from
the following

\begin{definition}
The \emph{Narrow-Band Mexican Needlet Whittle estimator }for the parameters $%
\vartheta =(\alpha ,G)$ is provided by%
\begin{equation*}
(\widehat{\alpha }_{J_{1};J_{L}}^{M},\widehat{G}_{J_{1};J_{L}}^{M}):=\arg
\min_{\alpha ,G}\sum_{j=J_{1}}^{J_{L}}\left[ \frac{\sum_{k}\beta _{jk;p}^{2}%
}{GK_{j}^{M}\left( \alpha \right) }-\sum_{k=1}^{N_{j}}\log \left( \frac{%
\beta _{jk;p}^{2}}{GK_{j}^{M}\left( \alpha \right) }\right) \right] \text{ ,}
\end{equation*}%
or equivalently:%
\begin{eqnarray}
\widehat{\alpha }_{J_{1},J_{L}}^{M} &=&\arg \min_{\alpha
}R_{J_{1};J_{L}}^{M}(\alpha ,\widehat{G}_{J_{1};J_{L}}^{M}(\alpha )),
\label{narrowest} \\
R_{J_{1};J_{L}}^{M}(\alpha ) &=&\left( \log \widehat{G}_{J_{1};J_{L}}^{M}(%
\alpha )+\frac{1}{\sum_{j=J_{0}}^{J_{L}}N_{j}}%
\sum_{J_{1}=J_{1}}^{J_{L}}N_{j}\log K_{j}^{M}\left( \alpha \right) \right) 
\text{ ,}  \notag
\end{eqnarray}%
where $0<J_{1}<J_{L}$ is chosen such that $B^{J_{L}}-B^{J_{1}}\rightarrow
\infty $ and%
\begin{equation}
B^{J_{1}}=B^{J_{L}}\left( 1-g\left( J_{L}\right) \right) \text{ , }%
J_{1}=J_{L}+\frac{\log \left( 1-g\left( J_{L}\right) \right) }{\log B}\text{
.}  \label{BJ1}
\end{equation}%
We choose $0<g\left( J_{L}\right) <1$ s.t. $\lim_{J_{L}\rightarrow \infty
}g\left( J_{L}\right) =0$ and $\lim_{J_{L}\rightarrow \infty }J_{L}g\left(
J_{L}\right) ^{\frac{3}{2}}=0$ .
\end{definition}

For notational simplicity $B^{J_{1}}$ is defined as an integer (if this
isn't the case, modified arguments taking integer parts are completely
trivial). For definiteness, we can take for instance $g\left( J_{L}\right)
=J_{L}^{-3}$ .

\begin{theorem}
\label{theonarrowband}Let $\widehat{\alpha }_{J_{L};J_{1}}$ defined as in (%
\ref{narrowest}). Then under Condition \ref{REGULNEED2} we have
\end{theorem}

\begin{equation*}
g\left( J_{L}\right) ^{\frac{1}{2}}B^{J_{L}}\left( \widehat{\alpha }%
_{J_{1};J_{L}}-\alpha _{0}\right) \overset{d}{\longrightarrow }\mathcal{N}%
\left( 0,\frac{\sigma _{0}^{2}(1+\widetilde{\tau })}{\Phi \left( B\right) }%
\right) \text{ ,}
\end{equation*}%
where%
\begin{equation*}
\Phi \left( B\right) :=\log ^{2}B\frac{B^{2}}{\left( B^{2}-1\right) ^{2}}%
\left( \frac{4}{\left( B^{2}-1\right) }+2\left( \frac{\log B-1}{\log B}%
\right) \right) \text{ .}
\end{equation*}

\begin{proof}
The proof is very similar to the full band case, hence we provide here just
the main differences. Consider:%
\begin{eqnarray*}
S_{J_{1};J_{L}}\left( \alpha \right) &=&\frac{d}{d\alpha }%
R_{J_{1};J_{L}}^{M}\left( \alpha \right) \text{ ;} \\
Q_{J_{1};J_{L}}^{M}\left( \alpha \right) &=&\frac{d^{2}}{d\alpha ^{2}}%
R_{J_{1};J_{L}}^{M}\left( \alpha \right) \text{ .}
\end{eqnarray*}%
Let 
\begin{equation*}
\overline{S}_{J_{1};J_{L}}^{M}\left( \alpha _{0}\right) =\frac{-1}{%
\sum_{j=J_{1}}^{J_{L}}N_{j}}\sum_{j=J_{1}}^{J_{L}}\frac{K_{j,1}^{M}\left(
\alpha _{0}\right) }{K_{j}^{M}\left( \alpha _{0}\right) }\sum_{k=1}^{N_{j}}%
\left( \frac{\beta _{jk;p}^{2}}{G\left( \alpha _{0}\right) K_{j}^{M}\left(
\alpha _{0}\right) }-\frac{\widehat{G}_{J_{1};J_{L}}^{M}\left( \alpha
_{0}\right) }{G_{0}}\right) \text{ }.
\end{equation*}%
Simple calculations based on Proposition \ref{prop27} lead to%
\begin{eqnarray*}
\left( \sum_{j=J_{1}}^{J_{L}}B^{2j}\right) ^{n} &=&\frac{B^{2n}}{\left(
B^{2}-1\right) ^{n}}\left( B^{2J_{L}}-B^{2\left( J_{1}-1\right) }\right)
^{n}+o\left( B^{2nJ_{L}}\right) \\
&=&B^{nJ_{L}}\text{ }+O\left( B^{2nJ_{L}}g\left( J_{L}\right) \right) \text{
,}
\end{eqnarray*}%
We have:%
\begin{equation*}
\lim_{J_{L}\rightarrow \infty }\frac{B^{J_{L}}}{J_{L}g\left( J_{L}\right) }%
\mathbb{E}\left( \overline{S}_{J_{0},J_{L}}^{M}\left( \alpha _{0}\right)
\right)
\end{equation*}%
\begin{eqnarray*}
&=&\lim_{J_{L}\rightarrow \infty }\frac{B^{J_{L}}}{J_{L}g\left( J_{L}\right) 
}\kappa \frac{I_{p,0}\left( \alpha _{0}+1\right) }{I_{p,0}\left( \alpha
_{0}\right) } \\
&&\times \left( \sum_{j=J_{1}}^{J_{L}}\log B^{j}\cdot B^{2j}\left(
B^{-j}-B^{-J_{L}}\left( \frac{B-1}{B}+\frac{g\left( J_{L}\right) }{B}\right)
\right) \right) +o_{J_{L}}\left( 1\right) \\
&=&\lim_{J_{L}\rightarrow \infty }-\kappa \frac{I_{p,0}\left( \alpha
_{0}+1\right) }{I_{p,0}\left( \alpha _{0}\right) }\frac{\log B}{B+1}%
+o_{J_{L}}\left( 1\right)
\end{eqnarray*}

As in full band case, we collect out all the covariance terms defined in
Lemma \ref{Lemmasumtau} and following Corollary \ref{cor28}, we have 
\begin{equation*}
Var\left( \overline{S}_{J_{L};J_{1}}\left( \alpha _{0}\right) \right) =\frac{%
\sigma _{0}^{2}\left( 1+\widetilde{\tau }\right) }{\left(
\sum_{j=J_{1}}^{J_{L}}B^{2j}\right) ^{3}}\left( Z_{J_{1};J_{L}}\left(
2\right) +o_{J_{L}}\left( 1\right) \right) \text{.}
\end{equation*}

After some manipulations, we have:%
\begin{equation*}
\frac{1}{B^{4J_{L}}}Z_{J_{1};J_{L}}\left( 2\right)
\end{equation*}%
\begin{eqnarray}
&=&\left( 1-\frac{\left( 1-g\left( J_{L}\right) \right) ^{2}}{B^{2}}\right)
^{2}-\frac{\left( B^{2}-1\right) ^{2}}{B^{4}}\left( 1-g\left( J_{L}\right)
\right) ^{2}\left( 1-\log _{B}\left( 1-g\left( J_{L}\right) \right) \right)
^{2}  \notag \\
&=&\left( \frac{B^{2}-1}{B^{2}}\right) ^{2}\Delta Z_{J_{L};J_{1}}\left(
g\left( J_{L}\right) \right) +O\left( g\left( J_{L}\right) ^{2}\right) 
\notag \\
&=&\left( \frac{B^{2}-1}{B^{2}}\right) ^{2}\left( \frac{4}{\left(
B^{2}-1\right) }+\left( 2-\frac{2}{\log B}\right) \right) g\left(
J_{L}\right) +O\left( B^{4J_{L}}g\left( J_{L}\right) ^{2}\right) \text{ ,}
\label{Zconti}
\end{eqnarray}

where%
\begin{eqnarray*}
\Delta Z_{J_{1};J_{L}}\left( g\left( J_{L}\right) \right) &=&\left( 1+\frac{%
4g\left( J_{L}\right) }{\left( B^{2}-1\right) }\right) -\left( 1-2g\left(
J_{L}\right) \right) \left( 1+\frac{1}{\log B}g\left( J_{L}\right) \right)
^{2} \\
&=&\left( \frac{4}{\left( B^{2}-1\right) }+\left( 2-\frac{2}{\log B}\right)
\right) g\left( J_{L}\right)
\end{eqnarray*}

Thus we have%
\begin{equation*}
Z_{J_{L};J_{1}}\left( 2\right) =B^{4J_{L}}\Phi \left( B\right) g\left(
J_{L}\right) \text{ }+O\left( B^{4J_{L}}g\left( J_{L}\right) ^{2}\right) 
\text{.}
\end{equation*}%
Note that $\Phi \left( B\right) >0$ for $B>1$.

Hence we have%
\begin{equation*}
Var\left( \overline{S}_{J_{1};J_{L}}^{M}\left( \alpha _{0}\right) \right)
=\sigma _{0}^{2}(1+\widetilde{\tau })\Phi \left( B\right) g\left(
J_{L}\right) B^{-2J_{L}}\text{ ,}
\end{equation*}%
Consider now $Q_{J_{L};J_{1}}\left( \alpha \right) $, which we rewrite as%
\begin{equation*}
Q_{J_{1};J_{L}}^{M}\left( \alpha \right) =\frac{Q_{num}\left( \alpha \right) 
}{Q_{den}\left( \alpha \right) }\text{ .}
\end{equation*}%
Following a procedure similar to Lemma \ref{Mconsistency2}, we have 
\begin{equation*}
Q_{num}\left( \alpha \right) =c_{B}^{2}G_{0}^{2}I_{p}\left( B,\alpha -\alpha
_{0}\right) Z_{J_{1};J_{L}}\left( s\right) \text{ ,}
\end{equation*}%
where $s=2\left( 1+\frac{\alpha -\alpha _{0}}{2}\right) $. Following (\ref%
{Zconti}), we obtain%
\begin{equation*}
Q_{num}\left( \alpha \right) =c_{B}^{2}G_{0}^{2}I\left( B,\alpha -\alpha
_{0}\right) \Phi \left( B,s\right) B^{2sJ_{L}}g\left( J_{L}\right) +O\left(
B^{2sJ_{L}}g\left( J_{L}\right) ^{2}\right) \text{ ,}
\end{equation*}%
where%
\begin{equation*}
\Phi \left( B,s\right) =\log ^{2}B\frac{B^{s}}{\left( B^{s}-1\right) ^{2}}%
\left( \frac{2sg\left( J_{L}\right) }{B^{s}-1}+\frac{s\log B-2}{\log B}%
\right) \text{.}
\end{equation*}

Finally, we obtain 
\begin{eqnarray*}
Q_{den}\left( \alpha \right) &=&c_{B}^{2}G_{0}^{2}I\left( B,\alpha -\alpha
_{0}\right) \left( \sum_{j=J_{1}}^{J_{L}}B^{sj}\right) ^{2} \\
&=&c_{B}^{2}G_{0}^{2}I_{p}\left( B,\alpha -\alpha _{0}\right)
B^{2sJ_{L}}+o\left( B^{2sJ_{L}}\right) \text{ .}
\end{eqnarray*}%
Hence 
\begin{equation*}
Q_{J_{1};J_{L}}^{M}\left( \alpha \right) =\Phi \left( B,s\right) g\left(
J_{L}\right) +O\left( B^{2sJ_{L}}g\left( J_{L}\right) ^{2}\right) \text{ ,}
\end{equation*}%
and, for the consistency of $\alpha ,$ we have 
\begin{equation*}
Q_{J_{1};J_{L}}^{M}\left( \overline{\alpha }\right) \rightarrow _{p}\Phi
\left( B\right) g\left( J_{L}\right) \text{ .}
\end{equation*}%
Thus%
\begin{equation*}
\left( \frac{\sigma _{0}^{2}(1+\widetilde{\tau })}{\Phi \left( B\right) }%
\right) ^{-\frac{1}{2}}g\left( J_{L}\right) ^{\frac{1}{2}}B^{J_{L}}\frac{%
\overline{S}_{J_{1};J_{L}}^{M}\left( \alpha _{0}\right) }{%
Q_{J_{1};J_{L}}^{M}\left( \overline{\alpha }\right) }\overset{d}{\rightarrow 
}\mathcal{N}\left( 0,1\right) \text{ ,}
\end{equation*}%
as claimed. Finally we can see that 
\begin{equation*}
g\left( J_{L}\right) ^{\frac{1}{2}}B^{J_{L}}\mathbb{E}\frac{\overline{S}%
_{J_{1};J_{L}}^{M}\left( \alpha _{0}\right) }{Q_{J_{1};J_{L}}^{M}\left( 
\overline{\alpha }\right) }=O\left( J_{L}\cdot g\left( J_{L}\right) ^{\frac{3%
}{2}}\right) \underset{J_{L\rightarrow \infty }}{\rightarrow }0\text{ .}
\end{equation*}
\end{proof}

\section{The Plug-in procedure \label{plugin}}

In this Section, we will present a plug-in estimation procedure for the
spectral parameter $\alpha _{0}$ based on the interaction of the approach
described  here and the one based upon standard needlets introduced in \cite%
{dlm2}. As already mentioned in the Introduction, there already exist in
literature Whittle-like estimators for spectral parameter based on spherical
harmonics and standard needlets. The former, although characterized by a
higher efficiency, can be affected by the presence of masked regions over
the sphere, common set-up in Cosmological investigations, because of the
lack of localization in the spatial domain. The latter, as one here
developed, is not altered by partially observed regions, paying the price of
a lower precision. Therefore, our aim is to show that, if $4p-\alpha _{0}>0\,
$, the spectral parameter estimator $\widehat{\alpha }_{J_{0,}J_{L}}^{M}$ is
more efficient with respect to the standard needlet estimator $\widehat{%
\alpha }_{J_{L}}$. First of all, observe that 
\begin{eqnarray*}
\lim_{J_{L}\rightarrow \infty }B^{2J_{L}}Var\left( \widehat{\alpha }%
_{J_{0,}J_{L}}^{M}-\alpha _{0}\right)  &=&\sigma _{0}^{2}\left( 1+\widetilde{%
\tau }\right) \frac{\left( B^{2}-1\right) ^{3}}{B^{4}\log ^{2}B}\text{ ,} \\
\lim_{J_{L}\rightarrow \infty }B^{2J_{L}}Var\left( \widehat{\alpha }%
_{J_{L}}-\alpha _{0}\right)  &=&\rho _{0}^{2}\frac{\left( B^{2}-1\right) ^{3}%
}{B^{4}\log ^{2}B}\,\ ,
\end{eqnarray*}%
see again \cite{dlm2}. We can therefore observe that for $4p-\alpha _{0}>0$, 
\begin{equation*}
\sigma _{1}^{2}<\rho _{0}^{2}\text{ ,}
\end{equation*}%
where $\sigma _{1}^{2}:=\sigma _{0}^{2}\left( 1+\widetilde{\tau }\right) $.
Consider that, for any fixed $p:4p>\alpha _{0}$, $\sigma _{0}^{2},$ which
does not depend on $B$, becomes, by the Stirling's formula, 
\begin{equation*}
\sigma _{0}^{2}\simeq \frac{2}{\sqrt{\left( \pi \left( 2p-\frac{\alpha _{0}}{%
2}\right) \right) }}\text{ .}
\end{equation*}%
We have that $\sigma _{0}^{2}$ is smaller than $1$ for $4p\gtrsim \alpha
_{0}-2$, while easy calculations show that $\widetilde{\tau }<1$. On the
other hand, as described in \cite{dlm2}, $\rho _{0}^{2}=\rho _{0}^{2}\left(
\alpha _{0},B\right) $ is decreasing on $B$ (see also Table 1): any attempt
to reduce its value will produce an increase of the variance due to the term 
$\left( B^{2}-1\right) ^{3}/B^{4}\log ^{2}B.$

\begin{center}
\begin{tabular}{l|c|c|c|c|c|c|}
\cline{2-7}
& \multicolumn{3}{|c|}{Standard Needlet-$\rho _{0}^{2}$} & 
\multicolumn{3}{|c|}{Mexican Needlets - $\sigma _{1}^{2}$} \\ \cline{2-7}
& $B=\sqrt[4]{2}$ & $B=\sqrt{2}$ & $B=2$ & $p=2$ & $p=3$ & $p=4$ \\ \hline
\multicolumn{1}{|l|}{$\alpha _{0}=2$} & 5.00 & 2.24 & 1.16 & 0.62 & 0.49 & 
0.42 \\ 
\multicolumn{1}{|l|}{$\alpha _{0}=3$} & 5.04 & 2.53 & 1.34 & 0.67 & 0.51 & 
0.43 \\ 
\multicolumn{1}{|l|}{$\alpha _{0}=4$} & 5.10 & 2.64 & 1.57 & 0.75 & 0.55 & 
0.45 \\ \hline
\end{tabular}

Table 1: Comparison among different values of the variances $\rho _{0}^{2}$
and $\sigma _{0}^{2}$.
\end{center}

Hence, the plug-in procedure can be implemented in two steps:

\begin{itemize}
\item First step: compute $\widehat{\alpha }_{J_{L}}$ in the standard
needlet framework.

\item Second step: if $p>\widehat{\alpha }_{J_{L}}/4$, compute $\widehat{%
\alpha }_{J_{0,}J_{L}}^{M}$ by the mexican needlet approach; otherwhise,
accept $\widehat{\alpha }_{J_{L}}$.
\end{itemize}

\appendix{}

\section{Auxiliary results: preliminaries}

The results collected in this section, provided by standard analytical
calculations, are here reported to explicit the structure and the behaviour
of the function $f_{p}\left( \cdot \right) $ defined in in (\ref{weightfunc}%
).

\begin{lemma}
\label{integrals}Let 
\begin{equation*}
W_{2a,b,s}=\int_{0}^{\infty }t^{2a}\exp \left( -bt^{2}\right) \log
^{s}\left( t\right) dt\text{ .}
\end{equation*}%
We have%
\begin{equation*}
W_{2a,b,0}=\frac{b^{-\left( a+\frac{1}{2}\right) }}{2}\Gamma \left( a+\frac{1%
}{2}\right) \text{ ,}
\end{equation*}%
\begin{equation*}
W_{2a,b,1}=\frac{b^{-\left( a+\frac{1}{2}\right) }}{4}\left[ \frac{d}{da}%
-\log b\right] \Gamma \left( a+\frac{1}{2}\right)
\end{equation*}%
and%
\begin{equation*}
W_{2a,b,2}=\frac{b^{-\left( a+\frac{1}{2}\right) }}{8}\left[ \frac{d^{2}}{%
da^{2}}-2\log b\frac{d}{da}+\log ^{2}b\right] \Gamma \left( a+\frac{1}{2}%
\right) \text{ .}
\end{equation*}
\end{lemma}

\begin{proof}
Standard calculations lead to 
\begin{eqnarray*}
W_{2a,b,0} &=&\int_{0}^{\infty }t^{2a}\exp \left( -bt^{2}\right) dt \\
&=&\frac{b^{-\left( a+\frac{1}{2}\right) }}{2}\int_{0}^{\infty }\left(
bt^{2}\right) ^{a-\frac{1}{2}}\exp \left( -bt^{2}\right) \left( 2btdt\right)
\\
&=&\frac{b^{-\left( a+\frac{1}{2}\right) }}{2}\Gamma \left( a+\frac{1}{2}%
\right) \text{ ;}
\end{eqnarray*}%
Similarly%
\begin{eqnarray*}
W_{2a,b,1} &=&\int_{0}^{\infty }t^{2a}\exp \left( -bt^{2}\right) \log tdt \\
&=&\frac{1}{2}\int_{0}^{\infty }t^{2a}\exp \left( -bt^{2}\right) \log \left(
bt^{2}\right) dt-\frac{\log b}{2}\int_{0}^{\infty }t^{2a}\exp \left(
-bt^{2}\right) dt \\
&=&\frac{b^{-\left( a+\frac{1}{2}\right) }}{4}\int_{0}^{\infty }x^{a-\frac{1%
}{2}}\exp \left( -x\right) \left[ \log x-\log b\right] dx \\
&=&\frac{b^{-\left( a+\frac{1}{2}\right) }}{4}\left[ \frac{d}{da}-\log b%
\right] \Gamma \left( a+\frac{1}{2}\right) \text{ ;}
\end{eqnarray*}%
\begin{eqnarray*}
W_{a,b,2} &=&\int_{0}^{\infty }t^{2a}\exp \left( -bt^{2}\right) \left( \log
t\right) ^{2}dt \\
&=&\frac{b^{-\left( a+\frac{1}{2}\right) }}{8}\int_{0}^{\infty }\left(
bt^{2}\right) ^{a-\frac{1}{2}}\exp \left( -bt^{2}\right) \left[ \left( \log
bt^{2}\right) ^{2}-2\log b\log bt^{2}+\log ^{2}b\right] 2btdt \\
&=&\frac{b^{-\left( a+\frac{1}{2}\right) }}{8}\left[ \frac{d^{2}}{da^{2}}%
-2\log b\frac{d}{da}+\log ^{2}b\right] \Gamma \left( a+\frac{1}{2}\right) 
\text{ .}
\end{eqnarray*}
\end{proof}

\begin{lemma}
\label{lemmasums}Let $f_{p}\left( \cdot \right) $ be defined as in (\ref%
{weightfunc}). Then we have 
\begin{equation}
\sum_{l\geq 1}f_{p}^{a}\left( \frac{l}{B^{j}}\right) l^{n}=\frac{B^{\left(
n+1\right) j}}{2a^{ap+\frac{n+1}{2}}}\Gamma \left( ap+\frac{n+1}{2}\right)
+o\left( B^{j\left( n+1\right) }\right) \text{ ;}  \label{sumvar}
\end{equation}%
Moreover, for $\Delta j\in \mathbb{Z}$, we have%
\begin{equation*}
\sum_{l\geq 1}f_{p}^{a_{1}}\left( \frac{l}{B^{j}}\right) f_{p}^{a_{2}}\left( 
\frac{l}{B^{j+\Delta j}}\right) l^{n}\text{ .}
\end{equation*}%
\begin{equation}
=\frac{B^{\left( n+1\right) j}\tau _{p,a_{1},a_{2}}\left( \Delta j\right) }{%
2\left( a_{1}+a_{2}\right) ^{\left( a_{1}+a_{2}\right) p+\frac{n+1}{2}}}%
\Gamma \left( \left( a_{1}+a_{2}\right) p+\frac{n+1}{2}\right) +o\left(
B^{\left( n+1\right) j}\right) \text{ ,}  \label{sumcov}
\end{equation}%
where%
\begin{equation*}
\tau _{p,a_{1},a_{2}}\left( \Delta j\right) =\left( \frac{a_{1}B^{\Delta
j}+a_{2}B^{-\Delta j}}{a_{1}+a_{2}}\right) ^{-\left( \left(
a_{1}+a_{2}\right) p+\frac{n+1}{2}\right) }B^{\Delta j\left( \left(
a_{1}-a_{2}\right) p+\frac{n+1}{2}\right) }\text{ .}
\end{equation*}
\end{lemma}

\begin{proof}
Observe that%
\begin{equation*}
\sum_{l\geq 1}f_{p}^{a_{1}}\left( \frac{l}{B^{j}}\right) f_{p}^{a_{2}}\left( 
\frac{l}{B^{j+\Delta j}}\right) l^{n}
\end{equation*}%
\begin{eqnarray*}
&=&\sum_{l\geq 1}\exp \left( -l^{2}\left( \frac{a_{1}}{B^{2j}}+\frac{a_{2}}{%
B^{2j+2\Delta j}}\right) \right) \left( \left( \frac{l}{B^{j}}\right)
^{2}\right) ^{a_{1}p}\left( \left( \frac{l}{B^{j+\Delta j}}\right)
^{2}\right) ^{a_{2}p}l^{n} \\
&=&\frac{B^{jn}}{B^{2a_{2}p\Delta j}}\sum_{l\geq 1}\exp \left( -\left( \frac{%
l}{B^{j}}\right) ^{2}\left( \frac{a_{1}B^{2\Delta j}+a_{2}}{B^{2\Delta j}}%
\right) \right) \left( \left( \frac{l}{B^{j}}\right) ^{2}\right) ^{\left(
a_{1}+a_{2}\right) p+\frac{n}{2}} \\
&=&\frac{B^{j\left( n+1\right) }}{2B^{2a_{2}p\Delta j}}\left( \frac{%
B^{2\Delta j}}{a_{1}B^{2\Delta j}+a_{2}}\right) ^{\left( a_{1}+a_{2}\right)
p+\frac{n+1}{2}}\int_{0}^{+\infty }x^{\left[ \left( a_{1}+a_{2}\right) p+%
\frac{n-1}{2}\right] }\exp \left( -x\right) dx+o\left( B^{\left( n+1\right)
j}\right) \\
&=&\frac{B^{j\left( n+1\right) }}{2}\frac{B^{2\Delta j\left( a_{1}p+\frac{n+1%
}{2}\right) }}{\left( a_{1}B^{2\Delta j}+a_{2}\right) ^{\left(
a_{1}+a_{2}\right) p+\frac{n+1}{2}}}\Gamma \left( \left( a_{1}+a_{2}\right)
p+\frac{n+1}{2}\right) +o\left( B^{\left( n+1\right) j}\right) \\
&=&\frac{B^{\left( n+1\right) j}}{2\left( a_{1}+a_{2}\right) ^{\left(
a_{1}+a_{2}\right) p+\frac{n+1}{2}}}\tau _{p,a_{1},a_{2}}\left( \Delta
j\right) \Gamma \left( \left( a_{1}+a_{2}\right) p+\frac{n+1}{2}\right)
+o\left( B^{\left( n+1\right) j}\right) \text{ .}
\end{eqnarray*}%
Fixing $\Delta j=0$, $a_{1}=a_{2}=a/2$, we obtain%
\begin{equation*}
\sum_{l\geq 1}f_{p}^{a}\left( \frac{l}{B^{j}}\right) l^{n}=\frac{B^{\left(
n+1\right) j}}{2a^{ap+\frac{n+1}{2}}}\Gamma \left( ap+\frac{n+1}{2}\right)
+o\left( B^{j\left( n+1\right) }\right) \text{ ,}
\end{equation*}%
as claimed.
\end{proof}

We now investigate the behaviour of the function $K_{j}^{M}\left( \alpha
\right) $ and its derivatives.

\begin{proposition}
\label{propsumKj} Let%
\begin{equation*}
I_{p,s}\left( \alpha \right) =\frac{2}{C_{B}}W_{4p+1-\alpha ,2,s}=\frac{2}{%
C_{B}}\int_{0}^{\infty }t^{4p+1-\alpha }e^{-2t^{2}}\left( \log t\right)
^{s}dt\text{ , }s=0,1,2\text{ .}
\end{equation*}%
Then we have%
\begin{eqnarray}
K_{j}^{M}\left( \alpha \right) &=&\left( I_{p,0}\left( \alpha \right)
+o_{j}(1)\right) B^{-\alpha j}  \label{KM0} \\
&=&\frac{2^{-\left( 2p-\frac{\alpha }{2}+1\right) }}{C_{B}}\Gamma \left(
2p+1-\frac{\alpha }{2}\right) B^{-\alpha j}\text{ ;}  \notag
\end{eqnarray}%
\begin{eqnarray}
K_{j,1}^{M}\left( \alpha \right) &:&=\frac{d}{d\alpha }K_{j}^{M}\left(
\alpha \right)  \notag \\
&=&-\left( j\log B+\frac{I_{p,1}\left( \alpha \right) }{I_{p,0}\left( \alpha
\right) }+o(1)\right) K_{j}^{M}\left( \alpha \right)  \label{KM1}
\end{eqnarray}%
\begin{align}
K_{j,2}^{M}\left( \alpha \right) & :=\frac{d^{2}}{d\alpha ^{2}}%
K_{j}^{M}\left( \alpha \right)  \notag \\
& =\left( j^{2}\log ^{2}B+2j\log B\frac{I_{p,1}\left( \alpha \right) }{%
I_{p,0}\left( \alpha \right) }+\frac{I_{p,2}\left( \alpha \right) }{%
I_{p,0}\left( \alpha \right) }+o(1)\right) K_{j}^{M}\left( \alpha \right) 
\text{ .}  \label{KM2}
\end{align}
\end{proposition}

\begin{proof}
These proofs follow the ones concerning the scalar needlet case (see \cite%
{dlm2}). We have indeed%
\begin{eqnarray*}
K_{j}^{M}\left( \alpha \right) &=&\frac{1}{C_{B}B^{2j}}\underset{l\geq 1}{%
\sum }\left( \frac{l}{B^{j}}\right) ^{4p}\exp \left( -2\left( \frac{l}{B^{j}}%
\right) ^{2}\right) \left( 2l+1\right) l^{-\alpha } \\
&=&B^{\left( 2-\alpha \right) j}\frac{2}{C_{B}B^{2j}}\underset{l\geq 1}{\sum 
}\left( \frac{l}{B^{j}}\right) ^{4p}\exp \left( -2\left( \frac{l}{B^{j}}%
\right) ^{2}\right) \frac{\left( l^{1-\alpha }+o\left( l^{1-\alpha }\right)
\right) }{B^{\left( 1-\alpha \right) j}} \\
&=&B^{-\alpha j}\frac{2}{C_{B}}W_{4p+1-\alpha ,2,0}+o_{j}\left( B^{-\alpha
j}\right) \\
&=&B^{-\alpha j}I_{p,0}\left( \alpha \right) +o_{j}\left( B^{-\alpha
j}\right) \text{ ,}
\end{eqnarray*}%
\begin{eqnarray*}
K_{j,1}^{M}\left( \alpha \right) &=&\frac{1}{C_{B}B^{2j}}\underset{l\geq 1}{%
\sum }\left( \frac{l}{B^{j}}\right) ^{4p}\exp \left( -2\left( \frac{l}{B^{j}}%
\right) ^{2}\right) \left( 2l+1\right) l^{-\alpha }\left( -\log l\right) \\
&=&-K_{j}^{M}\left( \alpha \right) \log B^{j}-\frac{2}{C_{B}}B^{-\alpha
j}\left( \int t^{4p+1-\alpha }e^{-2t^{2}}\log tdt+o_{j}(1)\right) \\
&=&-\left( j\log B+\frac{I_{p,1}\left( \alpha \right) }{I_{p,0}\left( \alpha
\right) }+o_{j}(1)\right) K_{j}^{M}\left( \alpha \right) \text{ ,}
\end{eqnarray*}%
\begin{eqnarray*}
K_{j,2}^{M}\left( \alpha \right) &=&\frac{1}{C_{B}B^{2j}}\underset{l\geq 1}{%
\sum }\left( \frac{l}{B^{j}}\right) ^{4p}\exp \left( -2\left( \frac{l}{B^{j}}%
\right) ^{2}\right) \left( 2l+1\right) l^{-\alpha }\log ^{2}l \\
&=&K_{j}^{M}\left( \alpha \right) \left( \log B^{j}\right) ^{2}+2B^{-\alpha
j}\log B^{j}\left( \frac{2}{C_{B}}\int t^{4p+1-\alpha }e^{-2t^{2}}\log
tdt+o(1)\right) \\
&&+B^{-\alpha j}\left( \frac{2}{C_{B}}\int t^{4p+1-\alpha }e^{-2t^{2}}\log
^{2}tdt+o(1)\right) \\
&=&\left( j^{2}\log ^{2}B+2j\log B\frac{I_{p,1}\left( \alpha \right) }{%
I_{p,0}\left( \alpha \right) }+\frac{I_{p,2}\left( \alpha \right) }{%
I_{p,0}\left( \alpha \right) }+o(1)\right) K_{j}^{M}\left( \alpha \right) 
\text{ .}
\end{eqnarray*}
\end{proof}

\begin{corollary}
\label{corkjratio}From Proposition \ref{propsumKj}, we have that: 
\begin{equation}
\frac{K_{j}^{M}\left( \alpha \right) }{K_{j}^{M}\left( \alpha _{0}\right) }%
=I_{p}\left( B,\alpha -\alpha _{0}\right) B^{\left( \alpha _{0}-\alpha
\right) j}+o\left( B^{\left( \alpha _{0}-\alpha \right) j}\right) \text{ ,}
\label{MKalpha}
\end{equation}%
where 
\begin{equation*}
I_{p}\left( B,\alpha -\alpha _{0}\right) :=\left( 2\left( 2p+1\right)
\right) ^{\frac{\alpha _{0}-\alpha }{2}}\text{ .}
\end{equation*}
\end{corollary}

\begin{proof}
The computation above shows that%
\begin{eqnarray*}
I_{p,0}\left( \alpha \right) &=&\frac{2}{C_{B}}W_{4p+1-\alpha ,2,0} \\
&=&\frac{2^{-\left( 2p-\frac{\alpha }{2}+1\right) }}{C_{B}}\Gamma \left(
2p+1-\frac{\alpha }{2}\right) \text{ ,}
\end{eqnarray*}%
and following (\ref{KM0}) 
\begin{eqnarray*}
\frac{K_{j}^{M}\left( \alpha \right) }{K_{j}^{M}\left( \alpha _{0}\right) }
&=&B^{\left( \alpha _{0}-\alpha \right) j}2^{\frac{\alpha _{0}-\alpha }{2}}%
\frac{\Gamma \left( 2p+1-\frac{\alpha }{2}\right) }{\Gamma \left( 2p+1-\frac{%
\alpha _{0}}{2}\right) }+o\left( B^{\left( \alpha _{0}-\alpha \right)
j}\right) \\
&=&B^{\left( \alpha _{0}-\alpha \right) j}\left( 2\left( 2p+1\right) \right)
^{\frac{\alpha _{0}-\alpha }{2}}+o\left( B^{\left( \alpha _{0}-\alpha
\right) j}\right) \text{ , }
\end{eqnarray*}%
as claimed.
\end{proof}

The next results follow strictly Proposition 27 in \cite{dlm2}, hence we
will report the statements, while we will omit the proofs.

\begin{proposition}
\label{prop27}Let $s>0$, $B>1$. Then%
\begin{equation}
\sum_{j=J_{0}}^{J_{L}}B^{sj}=\frac{B^{s}}{B^{s}-1}\left(
B^{sJ_{L}}-B^{s\left( J_{0}-1\right) }\right) \text{ ;}  \label{sum0}
\end{equation}%
\begin{eqnarray}
\sum_{j=J_{0}}^{J_{L}}B^{sj}\log B^{j} &=&\frac{B^{s}}{B^{s}-1}\log B\left(
\left( J_{L}-\frac{1}{B^{s}-1}\right) B^{sJ_{L}}\right.   \label{sum1} \\
&&-\left. \left( \left( J_{0}-1\right) -\frac{1}{B^{s}-1}\right) B^{s\left(
J_{0}-1\right) }\right)   \notag
\end{eqnarray}%
\begin{equation}
\sum_{j=J_{0}}^{J_{L}}B^{sj}\left( \log B^{j}\right) ^{2}  \label{sum2}
\end{equation}%
\begin{eqnarray*}
&=&\frac{B^{s}}{B^{s}-1}\left( \log B\right) ^{2}\left( \left( \left( J_{L}-%
\frac{1}{B^{s}-1}\right) ^{2}+\frac{B^{s}}{\left( B^{s}-1\right) ^{2}}%
\right) B^{sJ_{L}}\right.  \\
&&-\left. \left( \left( \left( J_{0}-1\right) -\frac{1}{B^{s}-1}\right) ^{2}+%
\frac{B^{s}}{\left( B^{s}-1\right) ^{2}}\right) B^{s\left( J_{0}-1\right)
}\right) 
\end{eqnarray*}
\end{proposition}

\begin{corollary}
\label{cor28}Let 
\begin{equation*}
V_{J_{0};J_{L}}\left( s\right) =\left( \sum_{j=J_{0}}^{J_{L}}B^{sj}\right)
\left( \sum_{j=J_{0}}^{J_{L}}B^{sj}\left( \log B^{j}\right) ^{2}\right)
-\left( \sum_{j=J_{0}}^{J_{L}}B^{sj}\log B^{j}\right) ^{2}\text{ .}
\end{equation*}%
The we have%
\begin{equation*}
V_{J_{0};J_{L}}\left( s\right) =\left( \frac{B^{s}\log B}{B^{s}-1}\right)
^{2}\left[ \frac{B^{s}}{\left( B^{s}-1\right) ^{2}}\left(
B^{sJ_{L}}-B^{s\left( J_{0}-1\right) }\right) ^{2}-B^{s\left(
J_{L}+J_{0}-1\right) }\left( J_{L}-J_{0}+1\right) ^{2}\right] \text{ .}
\end{equation*}%
Moreover if $J_{0}=-J_{L}$
\end{corollary}

\begin{equation*}
V_{J_{L}}\left( s\right) =\left( \frac{B^{s}\log B}{B^{s}-1}\right) ^{2}%
\left[ \frac{B^{s}}{B^{s}-1}\left( B^{sJ_{L}}-B^{s\left( -J_{L}-1\right)
}\right) ^{2}-\frac{1}{B^{s}}B^{s\left( 2J_{L}-1\right) }\left(
2J_{L}+1\right) ^{2}\right] \text{ ,}
\end{equation*}%
so that%
\begin{equation}
\lim_{J_{L}\rightarrow \infty }B^{-2sJ_{L}}V_{J_{L}}\left( s\right) =\log
^{2}B\frac{B^{3s}}{\left( B^{s}-1\right) ^{4}}\text{ .}  \label{normalZ}
\end{equation}

\section{Auxiliary Results: Covariance terms}

\begin{lemma}
\label{Lemmasumtau}Let $\tau _{B}\left( \Delta j\right) $ be defined as in (%
\ref{tau_def}). Hence we have for $4p-\alpha _{0}>0,$ $J_{0}<0$, 
\begin{equation}
\Sigma _{0}\left( J_{L}\right) :=\sum_{j=J_{0}}^{J_{L}}B^{2j}\sum_{\Delta
j=-J_{0}-j}^{J_{L}-j}\tau _{B}\left( \Delta j\right) B^{a_{0}\Delta j}=\frac{%
B^{2}}{B^{2}-1}\left( 1+\widetilde{\tau }_{0}\right) B^{2J_{L}}+o\left(
B^{2J_{L}}\right) \text{ ,}  \label{sumtau0}
\end{equation}%
\begin{equation*}
\Sigma _{1}\left( J_{L}\right) :=\sum_{j=J_{0}}^{J_{L}}B^{2j}\log
B^{j}\sum_{\Delta j=-J_{0}-j}^{J_{L}-j}\tau _{B}\left( \Delta j\right)
B^{a_{0}\Delta j}
\end{equation*}%
\begin{equation}
=\frac{B^{2}\log B}{B^{2}-1}\left( \left( 1+\widetilde{\tau }_{0}\right)
J_{L}-\left( \frac{1}{B^{2}-1}\left( 1+\widetilde{\tau }_{1}\right) \right)
\right) B^{2J_{L}}+o\left( B^{2J_{L}}\right) \text{ ,}  \label{sumtau1}
\end{equation}%
\begin{equation*}
\Sigma _{2}\left( J_{L}\right) :=\sum_{j=J_{0}}^{J_{L}}B^{2j}\log
^{2}B^{j}\sum_{\Delta j=-J_{0}-j}^{J_{L}-j}\tau _{B}\left( \Delta j\right)
B^{a_{0}\Delta j}
\end{equation*}%
\begin{equation}
=\frac{B^{2}\log ^{2}B}{B^{2}-1}\left( \left( 1+\widetilde{\tau }_{0}\right)
J_{L}^{2}-\left( \frac{2}{B^{2}-1}\left( 1+\widetilde{\tau }_{1}\right)
\right) J_{L}+\frac{B^{2}+1}{\left( B^{2}-1\right) ^{2}}\left( 1+\widetilde{%
\tau }_{2}\right) \right) B^{2J_{L}}+o\left( B^{2J_{L}}\right) \text{ ,}
\label{sumtau2}
\end{equation}%
where%
\begin{equation*}
\widetilde{\tau }_{0}:=\frac{2^{\left( 4p+1-\alpha _{0}\right) }}{\left(
B^{\left( 4p+2-\alpha _{0}\right) }-1\right) }\text{ ;}
\end{equation*}%
\begin{equation*}
\widetilde{\tau }_{1}:=2^{4p+1-\alpha _{0}}\frac{\left( B^{4p+4-\alpha
_{0}}-1\right) }{\left( B^{4p+2-\alpha _{0}}-1\right) ^{2}}\,\text{;}
\end{equation*}%
\begin{equation*}
\widetilde{\tau }_{2}:=2^{4p+1-\alpha _{0}}\left( \frac{W_{p}\left( B\right) 
}{\left( B^{4p-\alpha _{0}+1}-1\right) ^{3}}\right) \text{ ,}
\end{equation*}%
and%
\begin{equation*}
W_{p}\left( B\right) :=\frac{\left( B^{6}B^{4p-\alpha _{0}}\left( B^{\left(
P-1\right) }+1\right) +B^{4}B^{\left( 4p-\alpha _{0}\right) }\left(
B^{3}B^{4p-\alpha _{0}+1}-6\right) +B^{2}\left( B^{\left( 4p-\alpha
_{0}\right) }+1\right) +1\right) \allowbreak }{B^{2}+1}.
\end{equation*}

Moreover if we define 
\begin{equation*}
Z_{J_{L}}:=\Sigma _{0}\left( J_{L}\right) \Sigma _{2}\left( J_{L}\right)
-\Sigma _{1}^{2}\left( J_{L}\right) \text{ ,}
\end{equation*}

we have 
\begin{equation*}
\lim_{J_{L}=0}B^{-4J_{L}}Z_{J_{L}}:=\frac{B^{6}\log ^{2}B}{\left(
B^{2}-1\right) ^{4}}\left( 1+\widetilde{\tau }\right)
\end{equation*}%
where 
\begin{equation*}
\widetilde{\tau }:=\frac{1}{B^{2}}\left( \left( B^{2}+1\right) \left( 
\widetilde{\tau }_{0}+\widetilde{\tau }_{2}+\widetilde{\tau }_{0}\widetilde{%
\tau }_{2}\right) +2\widetilde{\tau }_{1}-\widetilde{\tau }_{1}^{2}\right) 
\text{ }
\end{equation*}
\end{lemma}

\begin{proof}
Let us call $P=\left( 4p+1-\alpha _{0}\right) $ and observe that:%
\begin{equation*}
\sum_{\Delta j=J_{0}-j}^{J_{L}-j}B^{\alpha _{0}\Delta j}\tau _{B}\left(
\Delta j\right) -1
\end{equation*}%
\begin{eqnarray*}
&=&\sum_{\Delta j=J_{0}-j}^{J_{L}-j}B^{\Delta j}\left[ \cosh \left( \Delta
j\log B\right) \right] ^{-P}-1=2^{P}\sum_{\Delta j=J_{0}-j}^{J_{L}-j}\frac{%
B^{\Delta j}}{\left( B^{\Delta j}+B^{-\Delta j}\right) ^{P}}-1 \\
&=&2^{P}\left( \sum_{\Delta j=J_{0}-j}^{-1}\frac{1}{\left( B^{\Delta j\left( 
\frac{P-1}{P}\right) }+B^{-\Delta j\left( \frac{P+1}{P}\right) }\right) ^{P}}%
+\sum_{\Delta j=1}^{J_{L}-j}\frac{1}{\left( B^{\Delta j\left( \frac{P-1}{P}%
\right) }+B^{-\Delta j\left( \frac{P+1}{P}\right) }\right) ^{P}}\right) 
\text{ ,}
\end{eqnarray*}%
where we have considered the case $J_{0}<0$ . Hence we have, from
Proposition \ref{prop27}%
\begin{eqnarray*}
\sum_{\Delta j=J_{0}-j}^{-1}\frac{1}{\left( B^{\Delta j\left( \frac{P-1}{P}%
\right) }+B^{-\Delta j\left( \frac{P+1}{P}\right) }\right) ^{P}} &\simeq
&\sum_{\Delta j=1}^{j-J_{0}}B^{-\left( P+1\right) \Delta j} \\
&=&\frac{1}{B^{\left( P+1\right) }-1}\left( 1-B^{-\left( P+1\right) \left(
j-J_{0}\right) }\right)
\end{eqnarray*}%
\thinspace while we have%
\begin{eqnarray*}
\sum_{\Delta j=1}^{J_{L}-j}\frac{1}{\left( B^{\Delta j\left( \frac{P-1}{P}%
\right) }+B^{-\Delta j\left( \frac{P+1}{P}\right) }\right) ^{P}} &\simeq
&\sum_{\Delta j=1}^{J_{L}-j}B^{-\left( P-1\right) \Delta j} \\
&=&\frac{1}{B^{\left( P-1\right) }-1}\left( 1-B^{-\left( P-1\right) \left(
J_{L}-j\right) }\right) \text{ ,}
\end{eqnarray*}%
Consider now%
\begin{equation*}
\sum_{j=J_{0}}^{J_{L}}B^{2j}\sum_{\Delta j=-J_{0}-j}^{J_{L}-j}B^{\Delta
j}\tau _{B}\left( \Delta j\right)
\end{equation*}%
\begin{equation*}
=\sum_{j=J_{0}}^{J_{L}}B^{2j}+\sum_{j=J_{0}}^{-1}B^{2j}\sum_{\Delta
j=J_{0}-j}^{-1}2^{P}B^{\Delta j\left( P+1\right)
}+\sum_{j=1}^{J_{L}}B^{2j}\sum_{\Delta j=1}^{J_{L}-j}2^{P}B^{-\Delta j\left(
P-1\right) }.
\end{equation*}%
We have, given that \ $P+1>0$, if $J_{0}<0$%
\begin{equation*}
\sum_{j=J_{0}}^{-1}B^{2j}\sum_{\Delta j=J_{0}-j}^{-1}B^{\Delta j\left(
P+1\right) }
\end{equation*}%
\begin{eqnarray*}
&=&\sum_{j=J_{0}}^{-1}B^{2j}\frac{1}{B^{\left( P+1\right) }-1}\left(
1-B^{-\left( P+1\right) \left( j-J_{0}\right) }\right) \\
&=&\frac{1}{B^{\left( P+1\right) }-1}\left( \frac{B^{2}}{B^{2}-1}\left(
B^{-2}-B^{2\left( J_{0}-1\right) }\right) -\frac{B^{\left( P+1\right) J_{0}}%
}{B^{P-1}-1}\left( B^{\left( 1-P\right) \left( J_{0}-1\right) }-B^{-\left(
1-P\right) }\right) \right) \\
&=&o\left( B^{2J_{L}}\right) \text{ .}
\end{eqnarray*}%
On the other hand,%
\begin{equation*}
\sum_{j=1}^{J_{L}}B^{2j}\sum_{\Delta j=1}^{J_{L}-j}B^{-\Delta j\left(
P-1\right) }
\end{equation*}%
\begin{eqnarray*}
&=&\sum_{j=J_{0}}^{J_{L}}B^{2j}\frac{1}{B^{\left( P-1\right) }-1}\left(
1-B^{-\left( P-1\right) \left( J_{L}-j\right) }\right) \\
&=&\frac{1}{B^{\left( P-1\right) }-1}\left( \frac{B^{2}}{B^{2}-1}B^{2J_{L}}-%
\frac{B^{\left( P+1\right) }}{B^{\left( P+1\right) }-1}B^{2J_{L}}\right)
+o\left( B^{2J_{L}}\right) \\
&=&\left( \frac{B^{2}}{\left( B^{2}-1\right) \left( B^{P+1}-1\right) }%
\right) B^{2J_{L}}+o\left( B^{2J_{L}}\right) \text{ .}
\end{eqnarray*}%
Hence we have 
\begin{equation*}
\sum_{j=J_{0}}^{J_{L}}B^{2j}\sum_{\Delta j=-J_{0}-j}^{J_{L}-j}B^{\Delta
j}\tau _{B}\left( \Delta j\right)
\end{equation*}%
\begin{eqnarray*}
&=&\frac{B^{2}}{B^{2}-1}B^{2J_{L}}\left( 1+\frac{2^{P}}{\left(
B^{P+1}-1\right) }\right) +o\left( B^{J_{L}}\right) \\
&=&\frac{B^{2}}{B^{2}-1}B^{2J_{L}}\left( 1+\widetilde{\tau }_{0}\right)
+o\left( B^{J_{L}}\right) \text{ ,}
\end{eqnarray*}%
Similar calculations lead to%
\begin{equation*}
\sum_{j=J_{0}}^{J_{L}}B^{2j}\log B^{j}\sum_{\Delta
j=-J_{0}-j}^{J_{L}-j}B^{\Delta j}\tau _{B}\left( \Delta j\right)
\end{equation*}%
\begin{equation*}
=\sum_{j=J_{0}}^{J_{L}}B^{2j}\log B^{j}+2^{P}\left(
\sum_{j=J_{0}}^{-1}B^{2j}\log B^{j}\sum_{\Delta j=J_{0}-j}^{-1}B^{\Delta
j\left( P+1\right) }+\sum_{j=1}^{J_{L}}B^{2j}\log B^{j}\sum_{\Delta
j=1}^{J_{L}-j}B^{-\Delta j\left( P-1\right) }\right) \text{ ,}
\end{equation*}%
where%
\begin{equation*}
\sum_{j=1}^{J_{L}}B^{2j}\log B^{j}\sum_{\Delta j=1}^{J_{L}-j}B^{-\Delta
j\left( P-1\right) }
\end{equation*}%
\begin{eqnarray*}
&=&\sum_{j=1}^{J_{L}}B^{2j}\log B^{j}\left( \frac{1}{B^{\left( P-1\right) }-1%
}\left( 1-B^{-\left( J_{L}-j\right) \left( P-1\right) }\right) \right) \\
&=&\frac{\log B}{B^{\left( P-1\right) }-1}B^{2J_{L}}\left( \frac{B^{2}\left(
B^{P-1}-1\right) }{\left( B^{2}-1\right) \left( B^{P+1}-1\right) }J_{L}-%
\frac{B^{2}\left( B^{2+\left( P+1\right) }-1\right) \left( B^{\left(
P-1\right) }-1\right) }{\left( B^{2}-1\right) ^{2}\left( B^{\left(
P+1\right) }-1\right) ^{2}}\right) +o\left( B^{2J_{L}}\right) \\
&=&\frac{B^{2}\log B}{B^{2}-1}B^{2J_{L}}\left( \frac{1}{\left(
B^{P+1}-1\right) }J_{L}-\frac{\left( B^{2+\left( P+1\right) }-1\right) }{%
\left( B^{2}-1\right) \left( B^{\left( P+1\right) }-1\right) ^{2}}\right)
+o\left( B^{2J_{L}}\right)
\end{eqnarray*}%
while, if $J_{0}<0$%
\begin{equation*}
\sum_{j=J_{0}}^{-1}B^{2j}\log B^{j}\sum_{\Delta j=J_{0}-j}^{-1}B^{\Delta
j\left( P+1\right) }=o\left( B^{2J_{L}}\right) \text{ .}
\end{equation*}%
We hence obtain%
\begin{equation*}
\sum_{j=J_{0}}^{J_{L}}B^{2j}\log B^{j}\sum_{\Delta
j=-J_{0}-j}^{J_{L}-j}B^{\Delta j}\tau _{B}\left( \Delta j\right)
\end{equation*}%
\begin{eqnarray*}
&=&\frac{B^{2}\log BB^{2J_{L}}}{B^{2}-1}\left( J_{L}\left( 1+\widetilde{\tau 
}_{0}\right) -\frac{1}{B^{2}-1}\left( 1+\frac{2^{P}\left( B^{2+\left(
P+1\right) }-1\right) }{\left( B^{\left( P+1\right) }-1\right) ^{2}}\right)
\right) +o\left( B^{2J_{L}}\right) \\
&=&\frac{B^{2}\log BB^{2J_{L}}}{B^{2}-1}\left( J_{L}\left( 1+\widetilde{\tau 
}_{0}\right) -\frac{1}{B^{2}-1}\left( 1+\widetilde{\tau }_{1}\right) \right)
+o\left( B^{2J_{L}}\right) \text{ .}
\end{eqnarray*}

Furthermore, we have:%
\begin{equation*}
\sum_{j=J_{0}}^{J_{L}}B^{2j}\log ^{2}B^{j}\sum_{\Delta
j=-J_{0}-j}^{J_{L}-j}B^{\Delta j}\tau _{B}\left( \Delta j\right)
\end{equation*}%
\begin{equation*}
=\sum_{j=J_{0}}^{J_{L}}B^{2j}\log ^{2}B^{j}+2^{P}\left(
\sum_{j=J_{0}}^{-1}B^{2j}\log ^{2}B^{j}\sum_{\Delta j=J_{0}-j}^{-1}B^{\Delta
j\left( P+1\right) }+\sum_{j=1}^{J_{L}}B^{2j}\log ^{2}B^{j}\sum_{\Delta
j=1}^{J_{L}-j}B^{-\Delta j\left( P-1\right) }\right) \text{ .}
\end{equation*}%
We observe that 
\begin{equation*}
\sum_{j=1}^{J_{L}}B^{2j}\log ^{2}B^{j}\sum_{\Delta j=1}^{J_{L}-j}B^{-\Delta
j\left( P-1\right) }
\end{equation*}%
\begin{eqnarray*}
&=&\frac{1}{B^{P-1}-1}\left( \frac{B^{2}\log ^{2}B}{B^{2}-1}B^{2J_{L}}\left(
\left( J_{L}-\frac{1}{B^{2}-1}\right) ^{2}+\frac{B^{2}}{\left(
B^{2}-1\right) ^{2}}\right) -B^{-\left( P-1\right)
J_{L}}\sum_{j=1}^{J_{L}}B^{\left( P+1\right) j}\log ^{2}B^{j}\right) \\
&=&\frac{B^{2}\log ^{2}B}{\left( B^{2}-1\right) }B^{2J_{L}}\left(
J_{L}^{2}\left( \frac{1}{\left( B^{\left( P+1\right) }-1\right) }\right) -2%
\frac{1}{\left( B^{2}-1\right) }J_{L}\left( \frac{\left( B^{2+\left(
P+1\right) }-1\right) }{\left( B^{\left( P+1\right) }-1\right) ^{2}}\right)
\right. \\
&&\left. +\left( \frac{B^{2}+1}{\left( B^{2}-1\right) ^{2}}\frac{W_{p}\left(
B\right) }{\left( B^{P+1}-1\right) ^{3}}\right) \right) +o\left(
B^{2J_{L}}\right) \\
&=&\frac{B^{2}\log ^{2}B}{\left( B^{2}-1\right) }B^{2J_{L}}\left( \widetilde{%
\tau }_{0}J_{L}^{2}-2\frac{\widetilde{\tau }_{1}}{\left( B^{2}-1\right) }%
J_{L}+\frac{B^{2}+1}{\left( B^{2}-1\right) ^{2}}\widetilde{\tau }_{2}\right)
+o\left( B^{2J_{L}}\right)
\end{eqnarray*}%
where%
\begin{equation*}
W_{p}\left( B\right) =\frac{\left( B^{6}B^{P-1}\left( B^{\left( P-1\right)
}+1\right) +B^{4}B^{\left( P-1\right) }\left( B^{3}B^{P}-6\right)
+B^{2}\left( B^{\left( P-1\right) }+1\right) +1\right) \allowbreak }{B^{2}+1}%
\text{ .}
\end{equation*}

On the other hand we have 
\begin{equation*}
\sum_{j=J_{0}}^{-1}B^{2j}\log ^{2}B^{j}\sum_{\Delta j=J_{0}-j}^{-1}B^{\Delta
j\left( P+1\right) }=o\left( B^{2J_{L}}\right) \text{ ,}
\end{equation*}%
so that%
\begin{equation*}
\sum_{j=J_{0}}^{J_{L}}B^{2j}\log ^{2}B^{j}\sum_{\Delta
j=-J_{0}-j}^{J_{L}-j}B^{\Delta j}\tau _{B}\left( \Delta j\right)
\end{equation*}%
\begin{equation*}
=\frac{B^{2}\log ^{2}B}{\left( B^{2}-1\right) }B^{2J_{L}}\left( \left( 1+%
\widetilde{\tau }_{0}\right) J_{L}^{2}-\frac{2\left( 1+\widetilde{\tau }%
_{1}\right) }{\left( B^{2}-1\right) }J_{L}+\frac{B^{2}+1}{\left(
B^{2}-1\right) ^{2}}\left( 1+\widetilde{\tau }_{2}\right) \right) +o\left(
B^{2J_{L}}\right) \text{ .}
\end{equation*}%
Hence we have, from Corollary \ref{cor28}, that 
\begin{eqnarray*}
Z_{J_{L}} &=&\frac{B^{4}\log ^{2}B}{\left( B^{2}-1\right) ^{2}}%
B^{4J_{L}}\left( \left( \left( 1+\widetilde{\tau }_{0}\right) J_{L}^{2}-%
\frac{2\left( 1+\widetilde{\tau }_{1}\right) }{\left( B^{2}-1\right) }J_{L}+%
\frac{B^{2}+1}{\left( B^{2}-1\right) ^{2}}\left( 1+\widetilde{\tau }%
_{2}\right) \right) \left( 1+\widetilde{\tau }_{0}\right) \right. \\
&&\left. -\left( J_{L}\left( 1+\widetilde{\tau }_{0}\right) -\frac{1}{B^{2}-1%
}\left( 1+\widetilde{\tau }_{1}\right) \right) ^{2}\right) \\
&=&\frac{B^{6}\log ^{2}B}{\left( B^{2}-1\right) ^{4}}B^{4J_{L}}\left(
1+\left( \frac{\left( B^{2}+1\right) }{B^{2}}\left( \widetilde{\tau }_{0}+%
\widetilde{\tau }_{2}+\widetilde{\tau }_{0}\widetilde{\tau }_{2}\right) +%
\frac{2\widetilde{\tau }_{1}-\widetilde{\tau }_{1}^{2}}{B^{2}}\right) \right)
\\
&=&\frac{B^{6}\log ^{2}B}{\left( B^{2}-1\right) ^{4}}B^{4J_{L}}\left( 1+%
\widetilde{\tau }\right) \text{ ,}
\end{eqnarray*}%
as claimed.
\end{proof}

\section{Auxiliary Results: Cumulants}

\begin{lemma}
\label{cumulants}Let $A_{j}$ and $B_{j}$ be defined as in (\ref{Aj_cum}) and
(\ref{Bj_cum}). As $J_{L}\rightarrow \infty $%
\begin{equation*}
\frac{1}{B^{4J_{L}}}cum\left\{
\sum_{l_{1}}(A_{j_{1}}+B_{j_{1}}),\sum_{l_{2}}(A_{j_{2}}+B_{j_{2}}),%
\sum_{l_{3}}(A_{j_{3}}+B_{j_{3}}),\sum_{l_{4}}(A_{j_{4}}+B_{j_{4}})\right\}
\end{equation*}%
\begin{equation*}
=O_{J_{L}}\left( \frac{J_{L}^{4}\log ^{4}B}{B^{2J_{L}}}\right) \text{ .}
\end{equation*}
\end{lemma}

\begin{proof}
It is readily checked (see also \cite{dlm}) that%
\begin{equation*}
cum\left\{ \widehat{C}_{l},\widehat{C}_{l},\widehat{C}_{l},\widehat{C}%
_{l}\right\} =O\left( l^{-3}l^{-4\alpha _{0}}\right) \text{ .}
\end{equation*}%
Let us compute:%
\begin{equation*}
C_{j_{1},j_{2},j_{3}j_{4}}^{4}=cum\left( \frac{\sum_{k}\beta _{j_{1}k;p}^{2}%
}{N_{j_{1}}G_{0}K_{j_{1}}^{M}\left( \alpha _{0}\right) },\frac{\sum_{k}\beta
_{j_{2}k;p}^{2}}{N_{j_{2}}G_{0}K_{j_{2}}^{M}\left( \alpha _{0}\right) },%
\frac{\sum_{k}\beta _{j_{3}k;p}^{2}}{N_{j_{3}}G_{0}K_{j_{3}}^{M}\left(
\alpha _{0}\right) },\frac{\sum_{k}\beta _{j_{4}k;p}^{2}}{%
N_{j_{4}}G_{0}K_{j_{4}}^{M}\left( \alpha _{0}\right) }\right)
\end{equation*}%
\begin{equation*}
=\left( \prod_{i=1}^{4}\frac{1}{N_{j_{i}}G_{0}K_{j_{i}}^{M}\left( \alpha
_{0}\right) }\right) cum\left( \sum_{k}\beta _{j_{1}k;p}^{2},\sum_{k}\beta
_{j_{2}k;p}^{2},\sum_{k}\beta _{j_{3}k;p}^{2},\sum_{k}\beta
_{j_{4}k;p}^{2}\right)
\end{equation*}%
\begin{equation*}
=\sum_{l_{1},l_{2},l_{3},l_{4}}\left( \prod_{i=1}^{4}\frac{f_{p}^{2}\left( 
\frac{l_{i}}{B^{j_{i}}}\right) \left( 2l_{i}+1\right) }{%
N_{j_{i}}G_{0}K_{j_{i}}\left( \alpha _{0}\right) }\right) cum\left( \widehat{%
C}_{l_{1}},\widehat{C}_{l_{2}},\widehat{C}_{l_{3}},\widehat{C}_{l_{4}}\right)
\end{equation*}%
\begin{equation*}
=\sum_{l}\left( 2l+1\right) ^{4}\left( \prod_{i=1}^{4}\frac{f_{p}^{2}\left( 
\frac{l_{i}}{B^{j_{i}}}\right) }{N_{j_{i}}G_{0}K_{j_{i}}\left( \alpha
_{0}\right) }\right) cum\left( \widehat{C}_{l},\widehat{C}_{l},\widehat{C}%
_{l},\widehat{C}_{l}\right) +o\left( B^{-4j}\right)
\end{equation*}%
\begin{equation*}
=O\left( \sum_{l}\left( \prod_{i=1}^{4}B^{\left( \alpha _{0-}2\right)
j_{i}}f_{p}^{2}\left( \frac{l_{i}}{B^{j_{i}}}\right) \right) B^{\left(
2-4\alpha _{0}\right) j}\left( l^{1-4\alpha _{0}}\right) \right)
\end{equation*}%
\begin{equation*}
=O\left( B^{-6j}\prod_{i=1}^{4}\delta _{j}^{j_{i}}\right) \text{ .}
\end{equation*}%
Then we have%
\begin{eqnarray*}
&&cum\left\{ \frac{\widehat{G}_{J_{0},J_{L}}^{M}(\alpha _{0})}{G_{0}},\frac{%
\widehat{G}_{J_{0},J_{L}}^{M}(\alpha _{0})}{G_{0}},\frac{\widehat{G}%
_{J_{0},J_{L}}^{M}(\alpha _{0})}{G_{0}},\frac{\widehat{G}_{J_{0},J_{L}}^{M}(%
\alpha _{0})}{G_{0}}\right\} \\
&=&O\left( \frac{1}{B^{8J_{L}}}%
\sum_{j_{1}j_{2}j_{3}j_{4}}N_{j_{1}}N_{j_{2}}N_{j_{3}}N_{j_{4}}C_{j_{1},j_{2},j_{3}j_{4}}^{4}\right)
\\
&=&O\left( \frac{1}{B^{8J_{L}}}\sum_{j}B^{2j}\right) =O\left(
B^{-6J_{L}}\right) \text{ .}
\end{eqnarray*}%
As in \cite{dlm} and \cite{dlm2}, the proof can be divided into 5 cases,
corresponding respectively to 
\begin{equation*}
\frac{1}{B^{4J_{L}}}cum\left\{
\sum_{j_{1}}A_{j_{1}},\sum_{j_{2}}A_{j_{2}},\sum_{j_{3}}A_{j_{3}},%
\sum_{j_{4}}A_{j_{4}}\right\} ,\frac{1}{B^{4J_{L}}}cum\left\{
\sum_{j_{1}}B_{j_{1}},\sum_{j_{2}}B_{j_{2}},\sum_{j_{3}}B_{j_{3}},%
\sum_{j_{4}}B_{j_{4}}\right\}
\end{equation*}%
\begin{equation*}
\frac{1}{B^{4J_{L}}}cum\left\{
\sum_{j_{1}}A_{j_{1}},\sum_{j_{2}}B_{j_{2}},\sum_{j_{3}}B_{j_{3}},%
\sum_{j_{4}}B_{j_{4}}\right\} ,\frac{1}{B^{4J_{L}}}cum\left\{
\sum_{j_{1}}A_{j_{1}},\sum_{j_{2}}A_{j_{2}},\sum_{j_{3}}B_{j_{3}},%
\sum_{j_{4}}B_{j_{4}}\right\}
\end{equation*}%
and%
\begin{equation*}
\frac{1}{B^{4J_{L}}}cum\left\{
\sum_{j_{1}}A_{j_{1}},\sum_{j_{2}}A_{j_{2}},\sum_{j_{3}}A_{j_{3}},%
\sum_{j_{4}}B_{j_{4}}\right\} \text{ ,}
\end{equation*}%
where we have used \ref{Aj_cum}, \ref{Bj_cum}. We have for instance%
\begin{eqnarray*}
&&\frac{1}{B^{4J_{L}}}cum\left\{
\sum_{j_{1}}A_{j_{1}},\sum_{j_{2}}A_{j_{2}},\sum_{j_{3}}A_{j_{3}},%
\sum_{j_{4}}A_{j_{4}}\right\} \\
&=&O\left( \frac{1}{B^{4J_{L}}}\sum_{j_{1},j_{2}j_{3},j_{4}}\prod_{i=1}^{4}%
\left( B^{2j_{i}}\log B^{j_{i}}\right) C_{j_{1},j_{2},j_{3}j_{4}}^{4}\right)
\\
&=&O\left( \frac{1}{B^{4J_{L}}}\sum_{j}B^{8j}\log ^{4}B^{j}B^{-6j}\right) =O(%
\frac{1}{B^{4J_{L}}}\sum_{j}\log ^{4}B^{j}B^{2j})=O(\frac{\log ^{4}B^{J_{L}}%
}{B^{2J_{L}}})\text{ ;}
\end{eqnarray*}%
and%
\begin{eqnarray*}
&&\frac{1}{B^{4J_{L}}}cum\left\{
\sum_{j_{1}}B_{j_{1}},\sum_{j_{2}}B_{j_{2}},\sum_{j_{3}}B_{j_{3}},%
\sum_{j_{4}}B_{j_{4}}\right\} \\
&=&\frac{1}{B^{4J_{L}}}\left\{ \sum_{j}B^{2j}\log B^{j}\right\}
^{4}cum\left\{ \frac{\widehat{G}_{J_{L}}(\alpha _{0})}{G_{0}},\frac{\widehat{%
G}_{J_{L}}(\alpha _{0})}{G_{0}},\frac{\widehat{G}_{J_{L}}(\alpha _{0})}{G_{0}%
},\frac{\widehat{G}_{J_{L}}(\alpha _{0})}{G_{0}}\right\} \\
&=&O\left( \log ^{4}B^{J_{L}}B^{-2J_{L}}\right) \text{ ;}
\end{eqnarray*}%
The proof for the remaining terms is entirely analogous, and hence
omitted.\bigskip
\end{proof}

\end{document}